\documentclass[11pt]{article}
\usepackage{amsmath}
\usepackage{url}
\usepackage{amsfonts}
\usepackage{amssymb}
\usepackage{amsthm}
\usepackage{mathrsfs} 
\usepackage{graphicx}

\newtheorem{theorem}{Theorem}[section]

\newtheorem{proposition}[theorem]{Proposition}
\newtheorem{definition}[theorem]{Definition}
\newtheorem{lemma}[theorem]{Lemma}

\newtheorem{remark}[theorem]{Remark}

\newcommand{\A}{\mathcal{A}}

\newcommand{\spam}{\mathop{\mathrm{span}}}

\newcommand{\nats}{\mathbb{N}}
\newcommand{\reals}{\mathbb{R}}
\newcommand{\RR}{\mathbb{R}}
\newcommand{\comps}{\mathbb{C}}
\newcommand{\M}{\mathbb{M}}
\newcommand{\calo}{\mathcal{O}}

\newcommand{\sphere}{\mathbb{S}^2}
\newcommand{\sph}{\mathbb{S}} 

\newcommand{\dif}{\mathrm{d}}

\newcommand{\vect}[1]{\mathbf{#1}}

\renewcommand{\b}{B}

\newcommand{\bfA}{\mathbf{A}}

\newcommand{\bff}{\mathbf{f}}
\newcommand{\bfy}{\mathbf{y}}

\renewcommand{\d}{\mathrm{dist}}
\newcommand{\comp}{\b^{c}}

\renewcommand{\A}{\mathsf{A}}
\newcommand{\C}{\mathsf{C}}

\newcommand{\J}{\mathcal{J}}

\renewcommand{\L}{\mathcal{L}}
\newcommand{\ns}[1]{\left|\!\left|\!\left| {#1}\right|\!\right|\!\right|}

\newcommand{\scrc}{{\mathscr C}}

\def\calc{{\mathcal C}}
\def\K{\mathrm{K}}
\def\coll{\K_{\Xi}}

\def\bfa{{\bf a}}

\def\bfc{{\bf c}}
\def\bfe{{\bf e}}

\def\bfy{{\bf y}}

\def\bfzero{{\bf 0}}
\def\achi{\check \chi}

\title{Localized bases for kernel spaces on the unit sphere 
     }
 
\author{E.~Fuselier\thanks{ Department of Mathematics, High Point
    University, High Point, NC 27262, USA.} \and T.~Hangelbroek\thanks{
    Department of Mathematics, University of Hawaii, Honolulu, HI
    96822, USA.  Research supported by grant DMS-1047694 from the
    National Science Foundation.}\and F. J.~Narcowich\thanks{ Department
    of Mathematics, Texas A\&M University, College Station, TX 77843,
    USA. Research supported by grant DMS-0807033 from the National
    Science Foundation.}\and J. D.~Ward\thanks{ Department of
    Mathematics, Texas A\&M University, College Station, TX 77843,
    USA. Research supported by grant DMS-0807033 from the National
    Science Foundation.}\and G. B~Wright\thanks{ Department of
    Mathematics, Boise State University, Boise, ID 83725, USA.
    Research supported by grants DMS-0934581and DMS-0540779 from the
    National Science Foundation.}  }

\begin{document}
\maketitle
\begin{abstract}
  Approximation/interpolation from spaces of positive definite or
  conditionally positive definite kernels is an increasingly popular
  tool for the analysis and synthesis of scattered data, and is
  central to many meshless methods. For a set of $N$ scattered sites,
  the standard basis for such a space utilizes $N$ \emph{globally}
  supported kernels; computing with it is prohibitively expensive for
  large $N$. Easily computable, well-localized bases, with
  ``small-footprint" basis elements -- i.e., elements using only a
  small number of kernels -- have been unavailable. Working on
  $\sphere$, with focus on the restricted surface spline kernels
  (e.g. the thin-plate splines restricted to the sphere), we construct
  easily computable, spatially well-localized, small-footprint, robust
  bases for the associated kernel spaces. Our theory predicts that
  each element of the local basis is constructed by using a
  combination of only $\mathcal{O}((\log N)^2)$ kernels, which makes
  the construction computationally cheap.  We prove that the new basis
  is $L_p$ stable and satisfies polynomial decay estimates that are
  stationary with respect to the density of the data sites, and we
  present a quasi-interpolation scheme that provides optimal $L_p$
  approximation orders. Although our focus is on $\sph^2$, much of the
  theory applies to other manifolds -- $\sph^d$, the rotation group,
  and so on. Finally, we construct algorithms to implement these
  schemes and use them to conduct numerical experiments, which
  validate our theory for interpolation problems on $\sphere$
  involving over one hundred fifty thousand data sites.
\end{abstract}

\pagestyle{myheadings}
\thispagestyle{plain}
\markboth{E. Fuselier, T. Hangelbroek, F. J. Narcowich, J. D. Ward,
  and G. B. Wright}{Localized Bases for Kernel Spaces}
  \newpage
\section{Introduction}\label{sec0}

Approximation/interpolation with positive definite or conditionally
positive definite kernels is an increasingly popular tool for
analyzing and synthesizing of scattered data and is central to many
meshless methods. The main difficulty in using this tool is that
well-localized bases with ``small-footprint" elements -- i.e.,
elements using only a small number of kernels -- have been
unavailable.  With this in mind, we have two main goals for this
paper.

The first is the theoretical development of small-footprint bases that
are well-localized spatially, for a variety of kernels. For important
classes of kernels on $\sph^2$, the theory itself predicts that a
basis element requires only $\calo(\log(N)^2)$ kernels, where $N$ is
the number of data sites.

Previous numerical experiments on data sets, with $N$ on the order of
a thousand, used ad-hoc techniques to determine the number of kernels
per basis element. The predictions of our theory, on the other hand,
have been verified numerically on $\sph^2$ for data sets with over a
hundred thousand sites.

Our second goal is to show how to easily and efficiently compute these
small-footprint, well-localized, robust bases for spaces associated
with restricted surface-spline kernels on the sphere $\sphere$. The
kernels in question are spherical basis functions having the form
\begin{equation}
  k_{m}(x,\alpha) := (-1)^m (1- x\cdot \alpha)^{m-1}
\log(1-x\cdot\alpha), 
\label{suface_spline}
\end{equation}
for $m = 2, 3, \dots$ (cf. \cite[Eqn.~3.3]{MNPW}). The
kernel spaces are denoted $S_m(\Xi)$ -- these are finite dimensional
spaces of functions obtained as linear combinations of $k_m$, sampled
at some (finite) set of nodes $\Xi\subset\sphere$, plus a spherical
polynomial $p$ of degree $m-1$, i.e. $\sum_{\xi\in\Xi} a_{\xi}
k_m(\cdot,\xi)+p(\cdot)$.  The coefficients involved satisfy the
simple side conditions given in \ref{SS_space}.

The Lagrange functions $\chi_\xi$, which interpolate cardinal
sequences: $\chi_{\xi}(\zeta)= \delta_{\xi,\zeta},\ \zeta\in\Xi$, form
a basis for $S_m(\Xi)$. Recently, it has been shown in \cite{HNW2},
for restricted surface splines, as well as many other kernels, that
these functions decay extremely rapidly away from $\xi$.  Thus,
$\{\chi_{\xi}\}_{\xi\in\Xi}$ forms a basis that is theoretically quite
good (sufficient to demonstrate that the Lebesgue constant is
uniformly bounded, among many other things). However, determining a
Lagrange basis function generally requires solving a full linear
system with at least $N:= \#\Xi$ unknowns, so working with this basis
directly is computationally expensive. In this paper we consider an
alternative basis: one that shares many of the nice properties of the
Lagrange basis, yet its construction is computationally cheap.

Here is what we would desire in an easily computed, robust basis
$\{b_\xi\}_{\xi\in \Xi}$ for $S_m(\Xi)$. Each basis function should be
highly localized with respect to the mesh norm $h:= \max_{\xi \in \Xi}
\text{dist}(x,\xi)$ of $\Xi$. Moreover, each should have a nearly
stationary construction. By this we mean that each basis element
$b_\xi$ is of the form $\sum_{\eta\in\Upsilon(\xi)}
A_{\xi,\eta}k_m(\cdot,\eta)+p_\xi$, where the coefficients
$A_{\xi,\eta}$ and the degree $m-1$ polynomial $p_\xi$ are completely
determined by $k_m$ and a small subset of centers Specifically, we
wish $b_\xi$ to satisfy the following requirements:
\begin{align*}
  \text{\phantom{i}i)}~~ &\#\Upsilon(\xi) = \mathrm{c}(N)\\
  \text{ii)}~~ &|b_\xi(x)| \le
  \sigma(r/h),\
  r:=\text{dist}(x,\xi).
\end{align*}
where the number of points influencing each basis function
$\mathrm{c}(N)$ is constant or slowly growing with $N$, and the
function $\sigma(\cdot)$ decays rapidly -- at an exponential rate
$\sigma(t) \le Ce^{-\nu|t|}$ or at least at a fast polynomial rate
$\sigma(t) \le C(1+|t|)^{-J}$. The B-spline basis, constructed from
the family of truncated power functions (i.e., using $(x-y)_+^m$ in
place of $k_m(x,y)$), is a model solution to the problem we consider.
 
{\bf Main results.}  The solution we present is to consider a basis of
``local Lagrange'' functions, which are constructed below in
Section~\ref{LLB}.  It has the following properties:

\begin{itemize}
\item \emph{Numerical Stability}. For any $J>2$, one can construct a
  numerically stable basis with decay $\sigma(t) \le C(1+|t|)^{-J}$.

\item \emph{Small footprint}. Each basis function is determined by a
relatively small set of centers: $\mathrm{c}(N)\le M\bigl(\log
N\bigr)^2$, where the constant $M$ is proportional to the square of
the rate of decay $J$: $M\propto J^2$.

\item \emph{$L_p$ stability}. The basis is stable in $L_p$: sequence norms
$\|c\|_{\ell_p}$ of the coefficients are comparable to $L_p$ norms of
the expansion $\sum_{\xi\in\Xi} c_{\xi} b_{\xi}$.

\item \emph{Near-best $L_\infty$ approximation}. For sufficiently large
  $J$, the operator $Q_{\Xi}f = \sum_{\xi\in\Xi} f(\xi) b_{\xi}$
  provides near-best $L_{\infty}$ approximation.
\end{itemize}

{\bf Preconditioners.} Over the years practical implementation of
kernel approximation has progressed despite the ill-conditioning of
kernel bases. This has happened with the help of clever numerical
techniques like multipole methods and other fast methods of evaluation
\cite{Greengard, BPT, Johnson}  
and often with the help of preconditioners 
\cite{BP, FGP,Kansa,Sloan}.  Many results already exist in the RBF
literature concerning preconditioners and ``better'' bases. For a good
list of references and further discussion, see
\cite{Fasshauer}. Several of these papers use local Lagrange functions
in their efforts to efficiently construct interpolants, but the number
of points chosen to localize the Lagrange functions are ad hoc and
seem to be based on experimental evidence. For example, Faul and
Powell, in \cite{FaulPowell}, devise an algorithm which converges to a
given RBF interpolant that is based on local Lagrange interpolants
using about thirty nearby centers. Beatson--Cherrie--Mouat, in
\cite{BCM}, use fifty local centers (p.~260, Table~1) in their
construction along with a few ``far away'' points to control the
growth of the local interpolant at a distance from the center. In
other work, Ling and Kansa \cite{Kansa} and co-workers have studied
approximate cardinal basis functions based on solving least squares
problems.

An offshoot of our results 
is a strategy for selecting centers for preconditioning (as in
\cite{FaulPowell} and \cite{BCM}) that scales correctly with the total
number of centers $N$. We demonstrate the power of this approach in
Section \ref{app}, where the local basis is used to successfully
precondition kernel interpolation problems varying in size by several
orders of magnitude.

{\bf Organization.}  We now sketch the outline of the remainder of the
article. In Section~\ref{Background} we give some necessary
background: in Section~\ref{Background_sphere} we treat analysis on
spheres and in Section~\ref{CPD_intro} we treat conditionally positive
definite kernels. Section~\ref{LLB} presents the construction of the
local Lagrange basis. Much of the remainder of the article is devoted
to proving that this basis has the desired properties mentioned
above. However, doing this will first require a thorough understanding
of the (full) Lagrange basis $\{\chi_\xi\}_{\xi\in\Xi}$, which we
study in detail in Sections \ref{S:LagrangeBasis} and
\ref{S:Coefficients}.

In Section \ref{S:LagrangeBasis} we consider the full Lagrange basis:
the stable, local bases constructed in \cite{HNW2}. We first
numerically exhibit the exponential decay of these functions away from
their associated center. A subsequent experiment shows that the
coefficients in the expansion $\chi_{\xi} = \sum_{\zeta\in\Xi}
A_{\xi,\zeta}k_m(\cdot,\zeta)$ have similar rapid decay. These
numerical observations confirm the theory in
Section~\ref{S:Coefficients}, where it is proven that the Lagrange
coefficients indeed decay quickly and stationarily with respect to $h$
as $\zeta$ moves away from $\xi$.

Section~\ref{S:better_basis} treats the main arguments of the
paper. This occurs roughly in three stages.
\begin{enumerate}
\item The decay of the Lagrange function coefficients indicates that
  {\em truncated} Lagrange functions $\widetilde \chi_{\xi} =
  \sum_{\zeta\in\Upsilon(\xi)} A_{\xi,\zeta}k_m(\cdot,\zeta)$ will be
  satisfactory.  But simply truncating causes the function to fall
  outside of the space (moment conditions for the coefficients are no
  longer satisfied), so it is necessary to adjust the coefficients
  slightly. The cost of this readjustment is related to the smallest
  eigenvalue of a certain Gram matrix: a symmetric positive definite
  matrix that depends on the set $\Upsilon(\xi)$ -- this is discussed
  in Section \ref{realigning}.
\item Section \ref{spheres_proj_lagrange} estimates the minimal
  eigenvalue of the Gram matrix, which is shown to be quite small
  compared to tail of the coefficients.  Although the resulting
  truncated, adjusted Lagrange function decays at a fast polynomial
  rate and requires $\mathcal{O}((\log N)^2)$ terms, it is still
  unsuitable because its construction requires the full expansion.
\item The local basis of Section \ref{LLB} has coefficients
  sufficiently close to those of $\widetilde \chi_{\xi}$ to guarantee
  that it too satisfies the above properties. This, the main theorem
  and its corollaries are given in Section~\ref{loc_lag_bases}.
\end{enumerate}
In Section~\ref{app} we demonstrate the effectiveness of using the
local basis to build preconditioners for large kernel interpolation
problems.
 
{\bf Generalization to other manifolds/kernels.} Finally, we note that
many of the results here can be demonstrated in far greater generality
with minimal effort: in particular, most results hold for Sobolev
kernels on manifolds (as considered in \cite{HNW}) and for many
kernels of polyharmonic and related type on two point homogeneous
spaces (as considered in \cite{HNW2}). To simplify our exposition, we
focus almost entirely on surface splines on $\sphere$. We will
include remarks discussing the generalizations as we go along.

\section{Background}\label{Background}
\subsection{The sphere}\label{Background_sphere}
We denote by $\sphere$ the unit sphere in $\reals^{3}$, and by $\mu$
we denote Lebesgue measure. The distance between two points, $x$ and
$\xi$, on the sphere is written $\d(x,\xi) := \arccos (x\cdot
\xi)$. The basic neighborhood is the spherical `cap' $B(\alpha,r):=
\{x\in\sphere: \d(x,\alpha)<r\}$. The volume of a spherical cap is
$\mu(B(\alpha,r)) = 2\pi(1-\cos r)$.

Throughout this article, $\Xi$ is assumed to be a finite set of
distinct nodes on $\sphere$ and we denote the number of elements in
$\Xi$ by $\#\Xi$.  The \emph{mesh norm} or f\emph{ill distance}, $h :=
h(\Xi,\sphere):= \max_{x \in \sphere}\d(x,\Xi),$ measures the density
of $\Xi$ in $\sphere$. The \emph{separation radius} is $q:= \frac12
\min_{\xi\ne \zeta} \text{dist}(\zeta,\xi)$, where $\xi,\zeta\in \Xi$,
and the \emph{mesh ratio} is $\rho_\Xi :=h/q$.

{\bf The Laplace-Beltrami operator and spherical coordinates}. Given a
north pole on $\sphere$, we will use the longitude $\theta_1\in
[0,2\pi)$ and the colatitude $\theta_2\in [0,\pi]$ as coordinates. The
Laplace--Beltrami operator is then given by
$$
\Delta = \frac{1}{\sin(\theta_2)}\frac{\partial}{\partial
  \theta_2}\sin(\theta_2) \frac{\partial}{\partial \theta_2}
+\frac{1}{\sin^2(\theta_2)}\frac{\partial^2}{\partial \theta_1^2}.
$$
For each $\ell\in \nats$, the eigenvalues of the negative of the
Laplace-Beltrami operator, $-\Delta$, have the form
$\nu_\ell:=\ell(1+\ell)$; these have multiplicity $2\ell+1$. For each
fixed $\ell$, the eigenspace $\mathcal{H}_{\ell}$ has an orthonormal
basis of $2\ell+1$ eigenfunctions, $\{Y_{\ell}^\mu\}_{\mu=-\ell}^{\ell}$,
the {\em spherical harmonics} of degree $\ell$. The space of spherical
harmonics of degree $\ell\le\sigma$ is $\Pi_\sigma = \bigoplus_{\ell
  \le \sigma} \mathcal{H}_\ell$ and has dimension
$(\sigma+1)^2$. These and are the basic objects of Fourier analysis on
the sphere.  In order to simplify notation, we often denote a generic
basis for $\Pi_\sigma$ as $(\phi_j)_{j=1\dots (\sigma+1)^2}$. We
deviate from this only when a specific basis of spherical harmonics is
required.

{\bf Covariant derivatives}. The second kind of operators are the
covariant derivative operators. These play a secondary role in this
article -- they are a construction useful for defining smoothness
spaces and for proving results about surface splines, but they play no
role in the actual implementation of the algorithms. We present some
useful overview here -- a more detailed discussion, where the relevant
concepts are developed for Riemannian manifolds (including $\sphere$)
is given in \cite[Section 2]{HNW}.  A more complete discussion is not
warranted here.
  
We consider tensor valued operators $\nabla^m $, where each entry
$(\nabla^m)_{(i_1,\dots,i_k)}$ is a differential operator of order $m$
-- each index $i_j = 1$ or $2$, corresponding to the variables $x_1=
\theta_1$ and $x_2= \theta_2$. They transform tensorially and are the
``invariant'' partial derivatives of order $m$. For $m=1$, $\nabla$
is the ordinary gradient, but when $m=2$, $\nabla^2$ is the invariant
Hessian. In spherical coordinates, $\nabla^2 u$ is a rank two
covariant tensor with four components:
\begin{gather*}
  (\nabla^2 f)_{1,1} = \csc^2(\theta_2) \frac{\partial^2f}{\partial
    \theta_1^2} +\cot(\theta_2) \frac{\partial f}{\partial \theta_2}, \quad
  (\nabla^2 f)_{2,2} = \frac{\partial^2f}{\partial \theta_2^2}
  \\
  (\nabla^2 f)_{1,2} = (\nabla^2 f)_{2,1} = \csc(\theta_2)\bigg(
  \frac{\partial^2 f}{\partial \theta_1\partial \theta_2} - \cot(\theta_2)
  \frac{\partial f}{\partial \theta_1}\bigg).
\end{gather*} 
For a discussion of covariant derivatives on a general compact,
$C^\infty$ manifold, we refer the reader to \cite[(2.7)]{HNW}.

Of special importance is the fact that at each point $x\in\sphere$
there is a natural inner product $\langle \cdot,\cdot \rangle$ on the
space of tensors. The inner product employs the inverse of the metric
tensor. For the sphere this is a diagonal $2\times2$ matrix with
entries:  and $g^{1,1}(x) = (\sin^2(\theta_2))^{-1}$ and $g^{2,2}(x) =1 $.
The general form of the inner product for tensors is
\[
\langle \nabla^m f, \nabla^m g\rangle_x = \sum_{i_1,\dots,i_m}
(\nabla^mf (x))_{i_1,\dots,i_m} (\nabla^m g(x))_{i_1,\dots,i_m}
g^{i_1,i_1}(x)\dots g^{i_m,i_m}(x).
\]
This gives rise to a notion of the pointwise size of the $m$th
derivative at $x$:
\[
|\nabla^m f(x)| := \sqrt{\langle \nabla^m f, \nabla^m f\rangle_x}.
\]

{\bf Smoothness spaces}. This allows us to construct Sobolev spaces.
For each $m$ and each measurable subset $\Omega\subset \sphere$, the
$L_2$ Sobolev norm is
$$\|f\|_{W_2^m(\Omega)} := \left(\sum_{k\le m}
  \int_{\Omega} |\nabla^k f(x)|^2\dif \mu(x) \right)^{1/2}.
$$

\subsection{Conditionally positive definite kernels and interpolation}
\label{CPD_intro}
Many of the useful computational properties of restricted surface
splines stem from the fact that they are conditionally positive
definite. 
We treat the topic of conditional positive definiteness for a general
kernel $k$.
\begin{definition}\label{cpd}
   A kernel $k$ is conditionally positive definite with respect to a
   finite dimensional space $\varPi$ if, for any set of $N$ distinct
   centers $\Xi$, the matrix $\coll := \bigl(k(\xi,\zeta) \bigr)_{\xi,
     \zeta\in\Xi}$ is positive definite on the subspace of all vectors
   $\bfa \in \comps^{N}$ satisfying $\sum_{\xi\in\Xi} a_{\xi} p(\xi) =
   0 $ for $p\in \varPi$.
\end{definition}

Let $(\varphi_j)_{j\in\nats}$ be a complete orthonormal basis of
continuous functions.
Consider a kernel
\begin{equation}
\label{kernel_expansion}
k(x,y):=\sum_{j\in \nats} \tilde{k}(j) \varphi_j(x)\overline{\varphi_j(y)}
\end{equation}
with coefficients $\tilde{k}\in \ell_2(\nats)$ so that all but
finitely many coefficients $\tilde{k}(j)$ are positive.  Then $k$ is
conditionally positive definite with respect to the (finite
dimensional) space $\varPi := \spam(\varphi_j \mid j\in \J),$ where
$\J = \{j\mid \tilde{k}(j) \le 0\}.$ Indeed,
$$
  \sum_{\xi\in\Xi} \sum_{\zeta\in\Xi}  a_{\xi}
  k(\xi,\zeta)\overline{
    a_{\zeta}}  =
  \sum_{j\in\nats}\tilde{k}(j)\sum_{\xi,\zeta\in\Xi} \alpha_{\xi} 
  \varphi_j(\xi)\overline{a_{\zeta}\varphi_j(\zeta) } =
  \sum_{j\notin \J} \tilde{k}(j) \|\alpha\varphi_j\|_{\ell_2 (\Xi)}^2 >0
$$
provided $\sum_{\xi} a_{\xi} \varphi_j(\xi) = 0$ for $j\in \J$ (i.e.,
satisfying $\tilde{k}(j)\le 0$).

Conditionally positive definite kernels are important for the
following interpolation problem.  Suppose $\Xi\subset \sphere$ is a
set of nodes on the sphere, $f:\sphere\rightarrow\mathbb{R}$ is some
target function, and $f\bigr|_{\Xi}$ are the samples of $f$ at the
nodes in $\Xi$.  We look for a function that interpolates this data
from the space
$$
S(k,\Xi):= S(k,\Xi,\varPi) : = 
	\left\{\sum_{\xi\in\Xi}a_{\xi} k(\cdot,\xi) \mid
	\sum_{\xi\in\Xi}a_{\xi}p(\xi) =0, \;
	\forall p\in\varPi\right\}
	+\varPi,
$$
Provided $\Xi\subset \sphere$ is unisolvent with respect to $\varPi$
(meaning that $p(\xi) = 0$ for all $\xi \in\Xi$ implies that $p=0$ for
any $p\in \varPi$), the unique interpolant from $S(k,\Xi)$ can be
written
$$
s(\cdot) = \sum_{\xi\in \Xi} a_{\xi} k(\cdot,\xi) + \sum_{j\in\J} c_j
\varphi_j(\cdot),
$$
where the expansion coefficients satisfy the (non-singular) linear
system of equations:
\begin{equation}\label{Collocation_System}
\begin{pmatrix}
\coll & \Phi\\
\Phi^T &0
\end{pmatrix}
\begin{pmatrix}
\bfa\\\bfc\end{pmatrix}
=
\begin{pmatrix}
\bff\\\bfzero\end{pmatrix},
\end{equation}
where $\coll = (k(\xi_i,\xi_j))$, $i,j=1,\ldots,N$, and $\Phi =
(\varphi_j(\xi_i))$, $i=1,\ldots,N$, $j\in \J$.  This interpolant
plays a dual role as the minimizer of the semi-norm $|\cdot|_{k}$
induced from the ``native space'' semi-inner product
\begin{equation}\label{NS_norm}
\left \langle u,v\right\rangle_{k} =
\left \langle  \sum_{j\in\nats}\hat{u}(j) \varphi_j, 
 \sum_{j\in\nats}\hat{v}(j) \varphi_j 
\right \rangle_{k}
:= \sum_{j\notin \J} \frac{\hat{u}(j) \overline{\hat{v}(j)} }{\tilde{k}(j)}.
\end{equation}
Namely, it is the interpolant to $\bff$ having minimal semi-norm
$|u|_k = \sqrt{\langle u,u\rangle_k}$.
 
\section{Constructing the local Lagrange basis}\label{LLB}
The restricted surface splines $k_m$ (see \eqref{suface_spline}) are
conditionally positive definite with respect to the space of spherical
harmonics of degree up to $m-1$, i.e. $\Pi_{m-1}$. The finite
dimensional spaces associated with these kernels are denoted as in the
previous section:
\begin{eqnarray}\label{SS_space}
 S_m(\Xi) &:=&S(k_m,\Xi,\Pi_{m-1})\nonumber \\
 &= &\left\{\textstyle{\sum_{\xi\in \Xi}}a_\xi k_m(\cdot,\xi)\mid \
   \sum_{\xi\in\Xi} a_\xi \phi(\xi)  = 0,\, \forall \phi \in \Pi_{m-1} \right\} + \Pi_{m-1}.
\end{eqnarray}

\noindent The goal of this section is to provide an easily
constructed, robust basis for $S_m(\Xi)$. The fundamental idea behind
building this basis is to associate with each $\xi\in\Xi$, a new basis
function that interpolates over a relatively small set of nodes a
function that is cardinal at $\xi$.

Specifically, let $\Upsilon(\xi)$ be the $n \ll N$ nearest neighbors to the node $\xi$, including the node $\xi$; see Figure \ref{fig:local_lagrange_illustration} for an illustration.   Then the new basis function associated with $\xi$ is given by
\begin{align}
\achi_{\xi}(\cdot) = \sum_{\zeta \in \Upsilon(\xi)} A_{\xi,\zeta} k_m(\cdot,\zeta) + \sum_{j=1}^{m^2} c_{\xi,j} \phi_j(\cdot), \label{eq:approx_lagrange}
\end{align}
where $\phi_j$ are a basis for the spherical harmonics of degree $\leq m-1$.  The coefficients $A_{\xi,\zeta}$ and $c_{\xi,j}$ are determined from the cardinal conditions
\begin{align}
\achi_{\xi}(\zeta) = 
\begin{cases}
1 & \text{if $\zeta=\xi$}, \\
0 & \text{if $\zeta \in \Upsilon(\xi)\setminus\xi$},
\end{cases}
\quad\text{and}\quad
\sum_{\zeta\in\Upsilon(\xi)} A_{\xi,\zeta} \phi_j(\zeta)=0.
\label{LL_conditions}
\end{align}
These coefficients can be determined by solving the (small) linear system
 \begin{equation}
\begin{pmatrix}
\K_{\Upsilon(\xi)} & \Phi\\
\Phi^T &0
\end{pmatrix}
\begin{pmatrix}
\bfA_{\xi}\\\bfc_{\xi}\end{pmatrix}
=
\begin{pmatrix}
\mathbf{y}_{\xi}\\\bfzero\end{pmatrix},
\label{LL_system}
\end{equation}
where $\mathbf{y}_{\xi}$ represents the cardinal data and the entries
of the matrix follow from \eqref{Collocation_System}.  We call
$\achi_{\xi}$ a \emph{local Lagrange function} about $\xi$.

\begin{figure}[thb]
\centering
\includegraphics[width=0.55\textwidth]{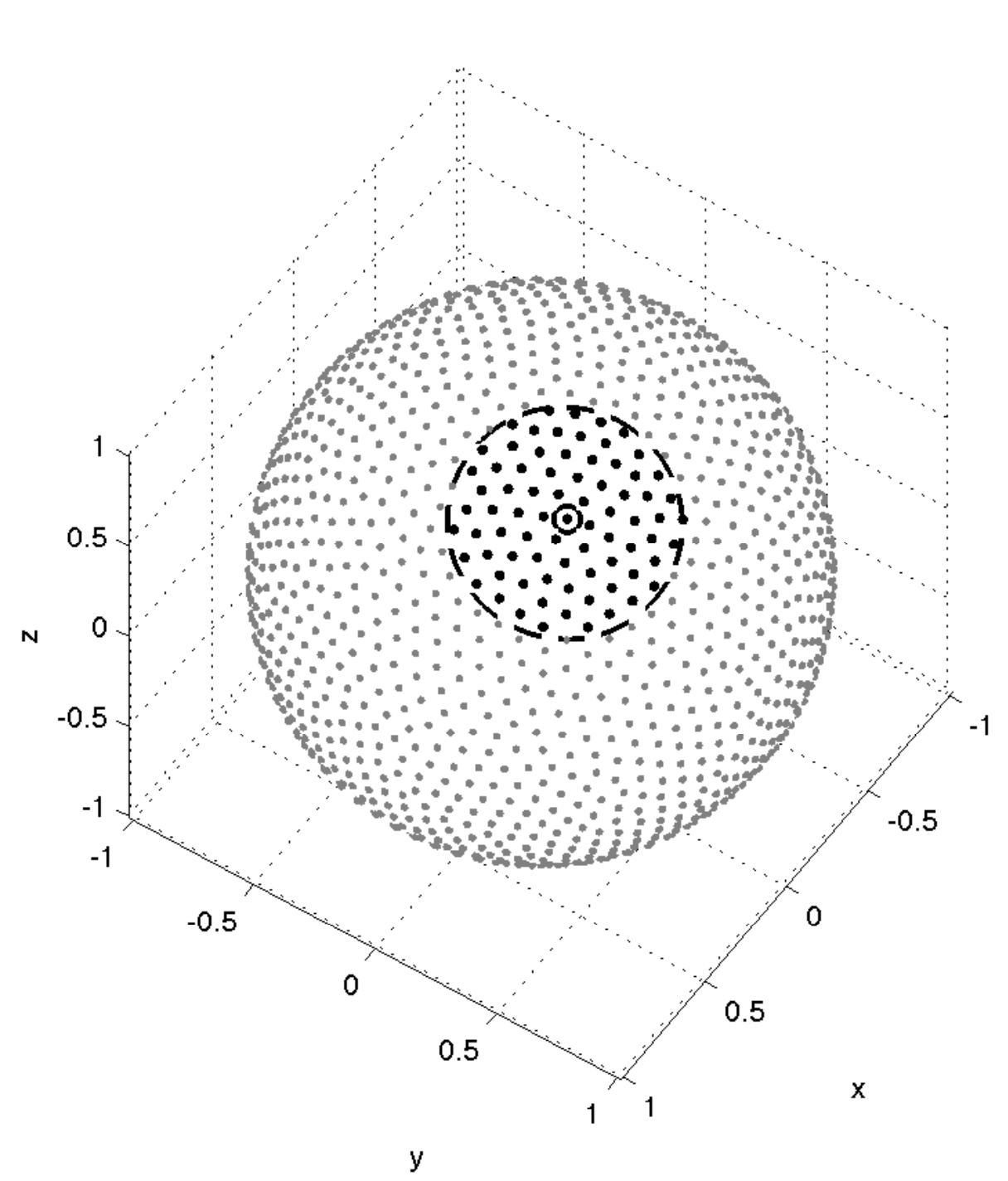} 
\caption{Illustration of the centers that make up the local Lagrange
  basis.  The solid gray and black spheres mark the set of $N$ nodes
  making up $\Xi$.  The solid black sphere with a circle around it
  marks the node $\xi$ where a local Lagrange function $\achi_{\xi}$
  is to be computed.  The the solid black spheres enclosed in the
  dashed circular line mark the set of $n = M(\log N)^2$ centers
  $\Upsilon(\xi)$ used to compute $\achi_{\xi}$.  For each
  $\xi\in\Xi$, a similar set $\Upsilon(\xi)$ is determined for
  computing $\achi_{\xi}$. \label{fig:local_lagrange_illustration}}
\end{figure}

The new basis for $S_m(\Xi)$ will consist of the collection of all the local Lagrange functions for the nodes in $\Xi$.  It will be shown in Section \ref{loc_lag_bases} that by choosing the number of nearest neighbors to each $\xi$ as $n = M(\log N)^2$ will give a basis with sufficient locality.  The choice of $M$ is related to the polynomial rate of decay of $\achi_{\xi}$ away from its center and \emph{a priori} estimates are given for $M$ in Section \ref{loc_lag_bases}.  However, in practice it will be sufficient to choose $M$ by tuning it appropriately to get the desired rate of decay.  

The exact details of the algorithm for constructing this basis then proceed as follows:  For each $\xi\in\Xi$ 
\begin{enumerate}
\item Find the $n = M(\log N)^2$ nearest neighbors to $\xi$,  $\Upsilon(\xi)$.
\item Construct $\achi_\xi$ according to the conditions \eqref{LL_conditions}, which amounts to solving the associated linear system \eqref{LL_system} and storing the coefficients $\bfA_\xi$, $\bfc_\xi$.
\end{enumerate}
We note each set $\Upsilon(\xi)$ can be determined in $\mathcal{O}(\log N)$ operations by using a KD-tree algorithm for sorting and searching through the nodes $\Xi$.  After the initial construction of the KD-tree, which requires $\mathcal{O}(N (\log N)^2)$, the construction of all the sets $\Upsilon(\xi)$ thus takes
$\mathcal{O}(N (\log N)^2)$ operations.

Before continuing, we note that our main results, given in Theorem
\ref{loc_lag_properties} and its corollaries, depend heavily on
properties that this local Lagrange basis inherits from the \emph{full
  Lagrange basis} $\{\chi_\xi\}_{\xi\in \Xi}$. Thus, much of what
follows is spent on developing a working understanding of the full
Lagrange basis and its connections to the local Lagrange basis. Even
though the local Lagrange basis is the focus of our work, we will
delay any further mention of $\{\achi_{\xi}\}_{\xi\in\Xi}$ until
Section \ref{loc_lag_bases}.

\section{The full Lagrange basis: numerical
  observations}\label{S:LagrangeBasis}

In this section we numerically examine a full Lagrange basis function
$\chi_\xi$ and its associated coefficients for the kernel
$k_2(x,\alpha) = (1-x\cdot\alpha)\log(1-x\cdot \alpha)$, the second
order restricted surface spline (also known as the thin plate spline)
on $\mathbb{S}^2$. First, we demonstrate numerically that $\chi_\xi$
decays exponentially away from its center. Secondly, we provide the
initial evidence that the Lagrange coefficients decay at roughly the
same rate, which is proved later in Theorem \ref{main}.

\begin{figure}[ht]
\centering
\includegraphics[height=90mm]{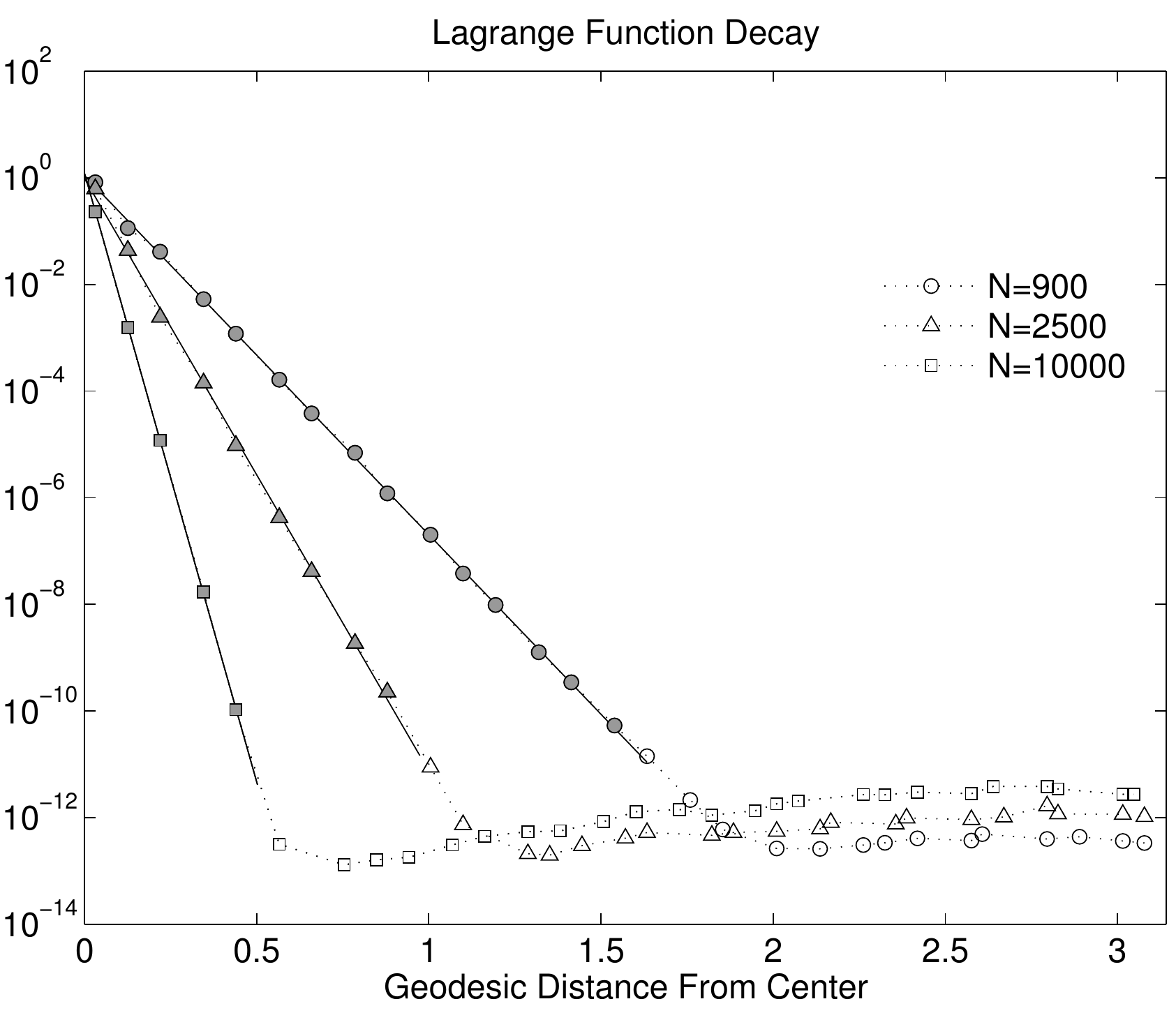}

\caption{Maximum latitudinal 
  values of the Lagrange function for the kernel $k_2(x,\alpha)$.
  This experiment was carried out in double precision arithmetic, and
  the plateau at roughly $10^{-11}$ occurs due to ill conditioning of
  the collocation matrices and truncation error.  }\label{ldTPS}
\end{figure}

\begin{figure}[h]
\centering
\includegraphics[height=90mm]{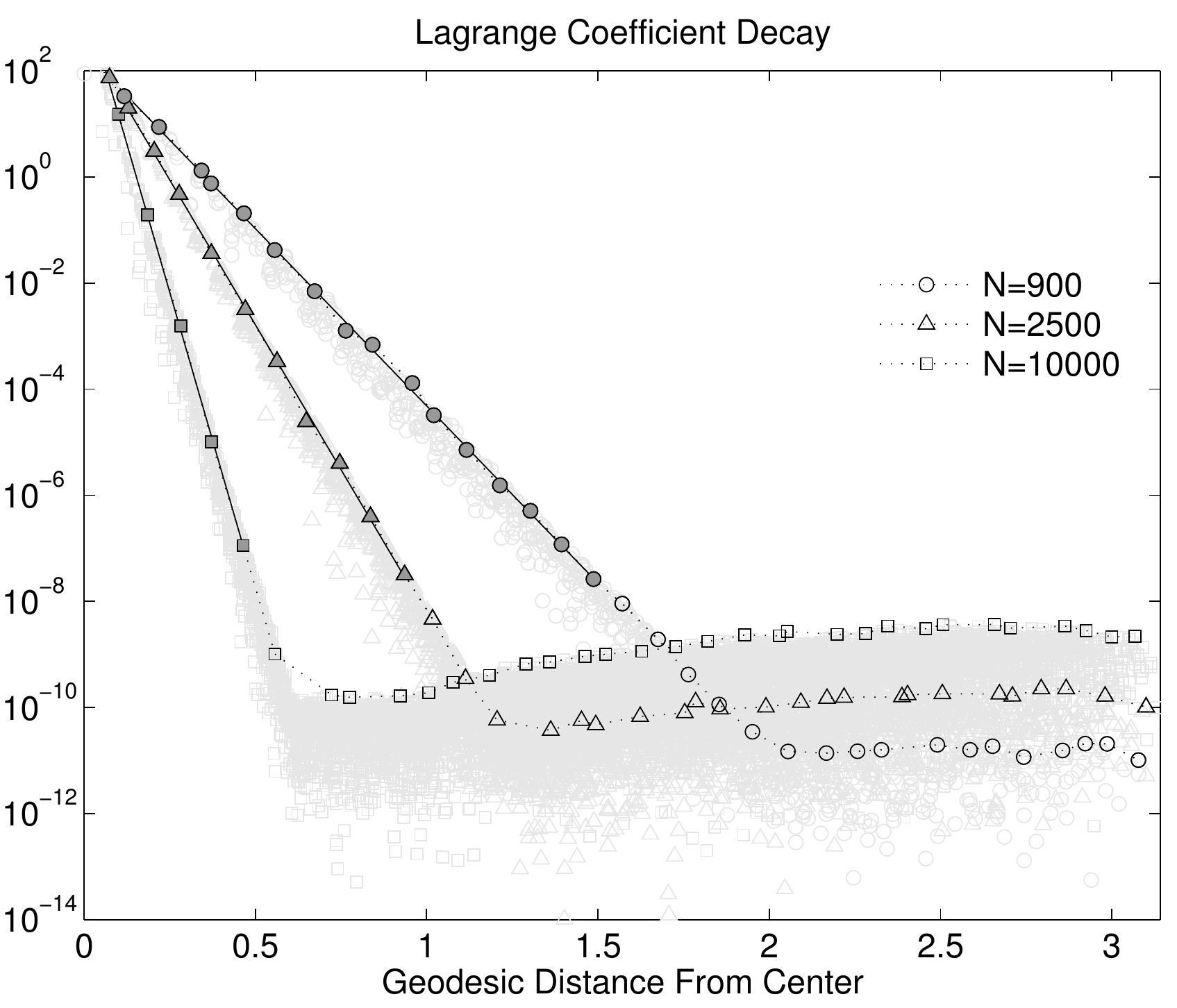}

\caption{Plot of coefficients for a Lagrange function in the kernel
  space
  $S(k_2,\Xi)$.
  This experiment was carried out in double precision arithmetic.
}\label{lcTPS}
\end{figure}


The full Lagrange function centered at $\xi$ takes the form
$$\chi_{\xi} = \sum_{\zeta\in\Xi} A_{\xi,\zeta}k(\cdot,\zeta)
+p_{\xi},$$ where $p_{\xi}$ is a degree $1$ spherical harmonic.  In
this example, we use the ``minimal energy points'' of Womersley for
the sphere -- these are described and distributed at the website
\cite{Wom}.\footnote{These point sets are used as benchmarks: each set
  of centers has a nearly identical mesh ratio, and the important
  geometric properties (e.g., fill distance and separation distance)
  are explicitly documented.} Because of the quasi-uniformity of the
minimal energy point sets, it is sufficient to consider the Lagrange
function $\chi_{\xi}$ centered at the north pole $\xi = (0,0,1)$.

Figure \ref{ldTPS} displays the maximal colatitudinal
values\footnote{The function $\chi_{\xi}$ is evaluated on a set of
  points $(\theta_1,\theta_2)$ with $n_0$ equispaced longitudes
  $\theta_1\in [0,2\pi]$ and $n_1$ equispaced colatitudes $\theta_2\in
  [0,\pi]$.} of $|\chi_{\xi}|$. Until a terminal value of roughly
$10^{-11}$, we clearly observe the exponential decay of the Lagrange
function, which follows
\begin{equation}\label{Locality}|\chi_{\xi}(x)| \le C_L
  \exp\left(-\nu_L \frac{\d(x,\xi)}{h}\right)
\end{equation}
(this ``plateau'' at $10^{-11}$ is caused by roundoff error -- see
Figure \ref{lcFREE_TPS}). The estimate \eqref{Locality} has in fact
been proven in \cite[Theorem 5.3]{HNW2}, where this and other analytic
properties of bases for $S_m(\Xi)$ were studied in detail. By fitting
a line to the data in Figure \ref{ldTPS} where the exponential decay
is evident, one can estimate the constants $\nu_L$ and $C_L$, which in
this case are quite reasonable. For example, the value of $\nu_L$,
which measures the rate of exponential decay, is observed to be close
to $1.35$ (see Table \ref{lagrange_stats}).

We can visualize the decay of the corresponding coefficients in the
same way. We again take the Lagrange function centered at the north
pole: for each $\zeta' \in\Xi$, the coefficient $|A_{\xi, \zeta'}|$ in
the expansion $\chi_{\xi} = \sum A_{\xi,\zeta}k(\cdot,\zeta)+ p_{\xi}$
is plotted with horizontal coordinate $\text{dist}(\xi,\zeta')$.  The
results for sets of centers of size $N= 900, 2500 $ and $10000$ are
given in Figure \ref{lcTPS}. The exponential decay seems to follow
$$
|A_{\zeta,\xi} |\le C_c q^{-2} \exp{\left(-\nu_c
    \frac{\d(\xi,\zeta)}{h}\right)}.
$$
Indeed, this is established later in Theorem \ref{main}. As before, we
can estimate the constants $\nu_c$ and $C_c$ for the decay of the
coefficients. Comparing Figures \ref{ldTPS} and \ref{lcTPS}, we note
that the coefficient plot is shifted vertically. This is a consequence
of the factor of $q^{-2}$ in the estimate (\ref{Coeff}) below.  Table
\ref{lagrange_stats} gives estimates for the constants $\nu_c$ and
$C_c$, along with the constants involved in the decay of the Lagrange
functions.

\begin{table}[ht]\label{lagrange_stats}
\begin{center}
\begin{tabular}{|c||c|c|c|c|c|c|}
\hline
$N$ & $h_X$ & $\rho_X$ & $\nu_L$ & $C_L$ & $\nu_c$ & $C_c$\\ 
\hline
\hline
400 & 0.1136 & 1.2930 & 1.1119 & 0.8382  & 1.0997 & 0.5402\\
\hline 
900 & 0.0874  & 1.5302 & 1.3556 & 1.0982& 1.3445 & 0.7554 \\ 
\hline
1600 & 0.0656 & 1.5333 & 1.3513 & 1.2170 & 1.3216 & 0.5946 \\
\hline
2500 & 0.0522 & 1.5278 & 1.3345 & 0.9618& 1.3117 & 0.5494 \\
\hline
5041 & 0.0365 & 1.5304 & 1.3395 & 1.1080  & 1.3158 & 0.6188\\
\hline 
10000 & 0.0260 & 1.5421 & 1.3645 & 1.1934 & 1.3369 & 0.7291\\
\hline
\end{tabular}
\caption{Estimates of the decay constants $\nu$ and $C$ for Lagrange functions and coefficients on the sphere using the kernel $k_2(x,\alpha)$, with relevant geometric measurements of the minimum energy node sets used.}
\end{center}
\end{table}

The perceived plateau present in the Lagrange function values as well
as the coefficients shown in Figures \ref{ldTPS} and \ref{lcTPS} is
due purely to round-off error related to the conditioning of kernel
collocation and evaluation matrices.  These results were produced
using double-precision (approximately 16 digits) floating point
arithmetic.  To illustrate this point, we plot the decay rate of the
Lagrange coefficients for the 900 and 1600 point node sets as computed
using high-precision (40 digits) floating point arithmetic in Figure
\ref{lcFREE_TPS}.  The figure clearly shows that the exponential decay
does not plateau, but continues as the theory predicts (see Theorem
\ref{main}).
\begin{figure}[h]
\centering
\includegraphics[height=90mm]{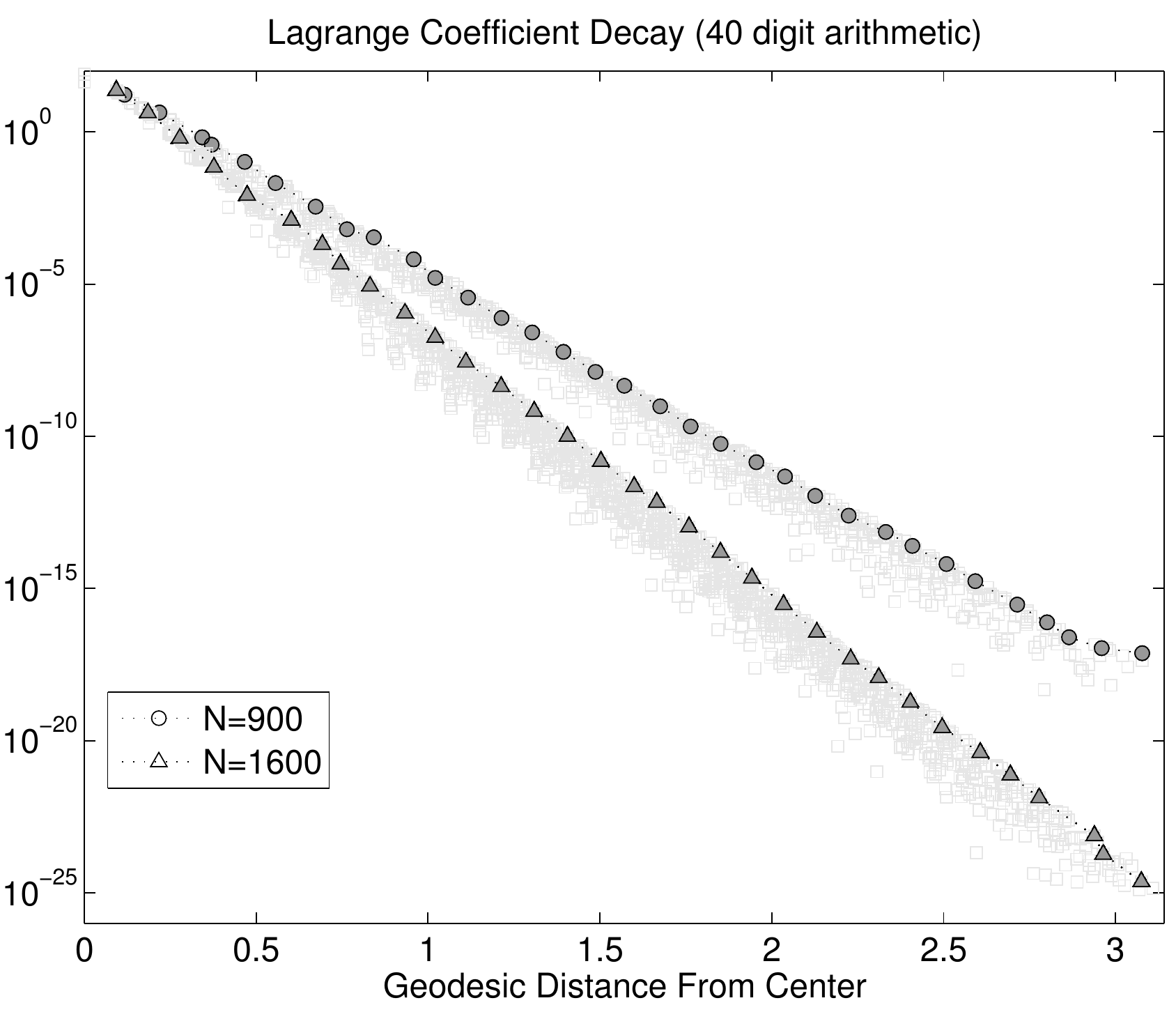}
\caption{Size of the coefficients for a Lagrange function in the kernel space $S(k_2,\Xi)$. This experiment was carried out in Maple with 40 digit arithmetic.
}\label{lcFREE_TPS}
\end{figure}

%
%

\section{Coefficients of the full Lagrange functions}\label{S:Coefficients}
In this section we give theoretical results for the coefficients
in the kernel expansion of Lagrange functions. In the first
part we give a formula relating the size of coefficients to native
space inner products of the Lagrange functions themselves
(this is  Proposition \ref{Lagrange_Coeffs_Formula}). 
We then obtain estimates for the restricted surface splines on $\sphere$, 
demonstrating the rapid, stationary decay of these coefficients.   

\subsection{Interpolation with conditionally positive definite
  kernels}\label{interpolation_kernels}

In this section we demonstrate that the Lagrange function coefficients
$A_{\xi,\zeta}$ can be expressed as a certain kind of inner product of
different Lagrange functions $\chi_{\xi}$ and $\chi_{\zeta}$.  Because
this is a fundamental result, we work in generality in this
subsection: the kernels we consider here are conditionally positive of
the type considered in Section \ref{CPD_intro}.

When $u, v\in S(k,\Xi)$ -- meaning that they  have the expansion
$
u=  \sum_{\xi\in{\Xi}} a_{1,\xi} k(\cdot,\xi) + p_u
$ 
and 
$v=  \sum_{\xi\in{\Xi}} a_{2,\xi} k(\cdot,\xi) + p_v
$
with coefficients $(a_{j,\xi})_{\xi\in\Xi} \perp (\varPi)|_{\Xi}$ for $j=1,2$ --
then the semi-inner product is
\begin{equation*}
  \left \langle u,v\right\rangle_{k} =
  \left \langle  \sum_{\xi\in{\Xi}} a_{1,\xi} k(\cdot,\xi) ,   
    \sum_{\xi\in{\Xi}} a_{2,\xi} k(\cdot,\xi) \right\rangle_{k} 
  =
  \sum_{\xi\in\Xi}\sum_{\zeta\in\Xi} a_{1,\xi}\overline{a_{2,\zeta}} k(\xi,\zeta).
\end{equation*}
(This follows directly from the definition (\ref{NS_norm}) coupled
with the observation that for $j\notin \J$, $\hat{u}(j) =
\sum_{\xi\in{\Xi}} a_{1,\xi}\tilde{k}(j)\phi_j(\xi) $ and $\hat{v}(j)
= \sum_{\xi\in{\Xi}} a_{2,\xi}\tilde{k}(j)\phi_j(\xi)$.)  We can use
this expression of the inner product to investigate the kernel
expansion of the Lagrange function.
\begin{proposition} \label{Lagrange_Coeffs_Formula} Let $k(\cdot,\xi)
  = \sum_{j\in \nats} \tilde{k}(j)
  \phi_j(\cdot)\overline{\phi_j(\xi)}$ be a conditionally
  positive definite kernel with respect to the space $\varPi
  =\spam_{j\in\J} \phi_j$, and let $\Xi$ be unisolvent for
  $\varPi$.  Then $\chi_{\eta}\in S(k,\Xi)$ (the Lagrange function
  centered at $\eta$) has the kernel expansion $\chi_{\eta}(x) =
  \sum_{\xi\in\Xi} A_{\eta,\xi} k(x,\xi) + p_{\zeta}$ with
  coefficients $$\bfA_{\eta} = (A_{\eta,\xi})_{\xi\in\Xi} =
  \bigl(\langle \chi_{\zeta}(x),
  \chi_{\eta}(x)\rangle_{k}\bigr)_{\xi\in\Xi}.$$
\end{proposition}
\begin{proof}
  Select two centers $\zeta,\eta\in\Xi$ with corresponding Lagrange
  functions $\chi_{\zeta}$ and $\chi_{\eta}\in S(k,\Xi)$.  Denote the
  collocation and auxiliary matrices, introduced in Section
  \ref{CPD_intro}, by $\coll = \bigl(k(\xi,\zeta)\bigr)_{\zeta,\xi}$
  and $\Phi = \bigl(\phi_j(\xi)\bigr)_{\xi,j}$.  Because
  $\bfA_{\zeta}$ and $\bfA_{\eta}$ are both orthogonal to
  $(\varPi)|_{\Xi}$, we have
$$
\langle \chi_{\zeta}, \chi_{\eta}\rangle_{k} =
\sum_{\xi_1\in\Xi}\sum_{\xi_2\in\Xi}
A_{\zeta,\xi_1}\overline{A_{\eta,\xi_2} }k(\xi_1,\xi_2) = \langle
\coll \bfA_{\zeta},\bfA_{\eta}\rangle_{\ell_2(\Xi)}.
$$

Now define $P:= \Phi (\Phi^* \Phi)^{-1} \Phi^*:\ell_2(\Xi)\to
(\Pi_{\J})|_{\Xi}\subset \ell_2(\Xi)$ to be the orthogonal projection
onto the subspace of samples of $\varPi$ on $\Xi$ and let $P^{\perp} =
\mathrm{Id}-P$ be its complement.  Then for any data $\bfy$,
(\ref{Collocation_System}) yields coefficient vectors $\bfA$ and
$\bfc$ satisfying $P^{\perp} \bfA = \bfA$ and $P^{\perp} \Phi \bfc =
\bfzero$, hence $P^{\perp} \coll P^{\perp} \bfA = P^{\perp}\coll\bfA =
P^{\perp} \bfy$.  Because $P^{\perp}:\ell_2(\Xi)\to\ell_2(\Xi)$ is
also an orthogonal projector, and therefore self-adjoint, it follows
that
\begin{eqnarray*}
\langle \chi_{\zeta}(x), \chi_{\eta}(x)\rangle_{k}  
&=&
\langle \coll \bfA_{\zeta},\bfA_{\eta}\rangle_{\ell_2(\Xi)}\\
&=&
\langle \coll  \bfA_{\zeta},P^{\perp}\bfA_{\eta}\rangle_{\ell_2(\Xi)}\\
&=&
\langle P^{\perp}\coll \bfA_{\zeta},\bfA_{\eta}\rangle_{\ell_2(\Xi)}\\
&=&
\langle P^{\perp}\bfe_{\zeta},\bfA_{\eta} 
\rangle_{\ell_2(\Xi)}.
\end{eqnarray*}
In the last line, we have introduced the sequence $\bfe_{\zeta} = (\delta_{\zeta,\xi})_{\xi\in\Xi}$ for
which $\coll \bfA_{\zeta} +p_{\zeta}|_{\Xi}=\bfe_{\zeta}$ which implies that  $P^{\perp}\coll \bfA_{\zeta} = P^{\perp} \bfe_{\zeta}$. Using once more the fact that $P^{\perp}$ is self-adjoint, and that $\bfA_{\eta}$ is in its
range, we have
$$\langle \chi_{\zeta}(x), \chi_{\eta}(x)\rangle_{k, \J}  =
\langle P^{\perp}\bfe_{\zeta},\bfA_{\eta}\rangle = \langle\bfe_{\zeta},  P^{\perp}\bfA_{\eta}\rangle
= \langle \bfe_{\zeta}, \bfA_{\eta}\rangle 
$$
and the lemma follows.
\end{proof}

The next result involves estimating the norms $\|\bfa \|_{\ell_2(\Xi)}$ and $\|\bfc \|_{\ell_2(\J)}$, where $\bfa$ and $\bfc$ are as in 
(\ref{Collocation_System}).
It will be useful later, when we discuss local Lagrange functions. The notation is the same as that used in the proof above. In addition, because $k$ is a conditionally positive definite kernel for $\varPi$, the matrix $P^\perp \coll P^\perp$ is positive definite on the orthogonal complement of the range of $\Phi$. 
We will let $\vartheta$ be the minimum eigenvalue of this matrix; that is, 
\[
\vartheta := \min_{\|P^\perp \alpha\|=1} \langle P^\perp \coll P^\perp \alpha,\alpha\rangle >0.
\]

\begin{proposition} \label{inverse_norm_est} 
Suppose $\bfa$ and $\bfc$ satisfy (\ref{Collocation_System}).  Let $G_\Xi = \Phi^\ast \Phi$. Then, 
\begin{equation*}
\|\bfa\|_{\ell_2(\Xi)} \le  \vartheta^{-1}\| \bfy \|_{\ell_2(\Xi)}\le  \vartheta^{-1}\sqrt{\# \Xi}\|\bfy\|_{\ell_\infty(\Xi)} 
\end{equation*}
and
\begin{eqnarray*}
 \|\bfc \|_{\ell_2(\J)} &\le&  2\|k\|_\infty \| G_\Xi^{-1}\|^{1/2} \vartheta^{-1}\#\Xi \|\bfy\|_{\ell_2(\Xi)}\\
 &\le&  2\|k\|_\infty \| G_\Xi^{-1}\|^{1/2} \vartheta^{-1}(\#\Xi)^{3/2} \|\bfy\|_{\ell_\infty(\Xi)}.
\end{eqnarray*}
\end{proposition}

\begin{proof}
From (\ref{Collocation_System}) and the fact that $P^\perp$ projects onto the orthogonal complement of the range of $\Phi$, we have that $P^\perp \coll P^\perp \bfa = P^\perp \bfy$ and that $P^\perp \bfa = \bfa$. Consequently, 
\[
\vartheta \| \bfa \|_{\ell_2(\Xi)}^2 =
\vartheta \| P^\perp \bfa \|_{\ell_2(\Xi)}^2 \le \langle P^\perp 
\coll P^\perp \bfa ,\bfa \rangle 
\le \|\bfa\|_{\ell_2(\Xi)} \| P^\perp \bfy\|_{\ell_2(\Xi)}.
\]
The bound on $\|\bfa\|_{\ell_2(\Xi)}$ follows immediately from this and the estimate $\|\bfy\|_{\ell_2(\Xi)} \le \sqrt{\# \Xi}\|\bfy\|_{\ell_\infty(\Xi)}$. 
To get the bound on $\|\bfc\|_{\ell_2(\Xi)}$, note that $\Phi \bfc =  P\bfy - P \coll \bfa$ and, hence, that
\begin{eqnarray*}
\| \Phi \bfc\|_{\ell_2(\Xi)} &\le& \|P\bfy \|_{\ell_2(\Xi)} + \|P\coll P^\perp \bfa\|_{\ell_2(\Xi)} \\
&\le& \|P\bfy \|_{\ell_2(\Xi)} + \vartheta^{-1}\|P\coll P^\perp\| \|P^\perp \bfy\|_{\ell_2(\Xi)}
\end{eqnarray*}
We also have that 
$\|\Phi \bfc\|_{\ell_2(\Xi)}^2 = \langle \Phi^\ast \Phi \bfc ,\bfc \rangle \ge \lambda_{min}(\Phi^\ast \Phi) \| \bfc \|_{\ell_2(\J)}^2$.  
However, $\lambda_{min}(\Phi^\ast \Phi) = \| (\Phi^\ast \Phi)^{-1}\|^{-1}$, which implies that 
\[
\| \bfc \|_{\ell_2(\J)} \le \| (\Phi^\ast \Phi)^{-1}\|^{1/2}  \| \Phi \bfc \|_{\ell_2(\Xi)}= \| G_\Xi^{-1}\|^{1/2}  \| \Phi \bfc \|_{\ell_2(\Xi)}
\] 
Next, note that the following inequalities hold:  $\|P\coll P^\perp\| \le \|\coll \| \le \#\Xi \|k\|_\infty$, 
$\|P\bfy\|_{\ell_2(\Xi)},\|P^\perp\bfy\|_{\ell_2(\Xi)}\le \|\bfy\|_{\ell_2(\Xi)}\le \sqrt{\# \Xi}  \|\bfy\|_{\ell_\infty(\Xi)}$, and 
$\frac{\#\Xi \|k\|_\infty}{\vartheta}\ge 1$. Applying these to the inequality
\[
\| \bfc \|_{\ell_2(\J)} \le \| G_\Xi^{-1}\|^{1/2} \big(\|P\bfy\|_{\ell_2(\Xi)}+ \#\Xi \|k\|_\infty \vartheta^{-1} \|P^\perp \bfy\|_{\ell_2(\Xi)}\big)
\]
then yields the desired bound on $\|\bfc\|_{\ell_2(\J)}$, completing the proof.
\end{proof}

\subsection{Estimating Lagrange function coefficients}\label{S:Estimating}
In \cite{HNW2}, it has been shown that Lagrange functions for
restricted surface splines decay exponentially fast away from the
center. We can use these decay estimates in conjunction with
Proposition \ref{Lagrange_Coeffs_Formula} to estimate the decay of the
coefficients $|A_{\xi,\zeta}|$.

Recall that the eigenvalues of $-\Delta$ are $\lambda_\ell =
\ell(\ell+1)$. Let $Q(z) := \Pi_{\nu=1}^m(z-\lambda_{\nu-1}) =
\sum_{\nu=1}^m b_\nu z^\nu$. The kernel $k_m:\sphere\times\sphere\to
\reals$ has the expansion
\[
k_m(x, \alpha) = \sum_{\ell=0}^\infty \tilde k(\ell) \sum_{\mu=-\ell}^\ell Y_{\ell}^\mu (x)Y_\ell^\mu(\alpha),
\]
where,  for $\ell\ge m$, $\tilde k(\ell) =C_mQ(\lambda_\ell)^{-1}$, with $C_m=2^{m+1}\pi \Gamma(m)^2 $ \cite[Eqn.~3.3]{MNPW}.   From the expansion, one sees that $k_m$ is conditionally positive definite with respect to $\Pi_{m-1}$. Kernels such as $k_m$ are said to be of {\em polyharmonic or related type}; they have been
studied in \cite{HNW2}.  The kernel $k_m$ acts as the Green's function for the elliptic operator
$\L_m:= C_m^{-1} Q(-\Delta)$ (cf. \cite[Example 3.3]{HNW2}), in the sense that
$$
f = \int_{\M} k_m(\cdot,\alpha) \L_m\bigl[ f(\alpha)
-p_f(\alpha)\bigr] \dif \alpha+p_f,
$$
where $p_f$ is the orthogonal projection of $f$ onto $\Pi_{m-1}$.

{\bf The native space ``inner product" on subsets.}  In \cite{HNW2} it
was shown that for any $k\in \nats$, the operator
$(\nabla^k)^*\nabla^k$ (which involves $(\nabla^k)^*$ the adjoint --
with respect to the $L_2(\sphere)$ inner product -- of the covariant
derivative operator $\nabla^k$ which was introduced in
Section~\ref{Background_sphere}) can be expressed as $\sum_{\nu=0}^k
d_{\nu} \Delta^{\nu}$ with $d_k=(-1)^k$. Consequently, any operator of
the form $\sum_{j=0}^k c_{j} (\nabla^{j})^* \nabla^j$ can be expressed
as $\sum_{\nu=0}^k d_{\nu} \Delta^{\nu}$ with $d_k=(-1)^kc_k $ and
vice-versa:
\begin{gather}\label{invariance}
  \forall  (d_0,\dots d_m)
  \  \exists (c_0,\dots,c_m) 
  \text{ with }d_m = (-1)^m c_m\nonumber \\
  \text{ and }  \sum_{\nu=0}^m d_\nu 
  \Delta^\nu = \sum_{j=1}^m c_j(\nabla^j)^*\nabla^j. 
\end{gather}
Because $\L_m=C_m^{-1}Q(-\Delta)$, it follows  that $\L_m=
\sum_{j=0}^m c_j(\nabla^j)^*\nabla^j$, with $c_m=C_m^{-1}$, and so the native
space semi-inner product, introduced in (\ref{NS_norm}), can be
expressed as
$$
\langle u,v\rangle_{k_m} = \langle \L_m u, v\rangle_{L_2(\sphere)}
 = 
 \int_{\sphere} \beta(u,v)_x \dif \mu(x)
$$ 
with $\beta(u,v)_x = \sum_{k=0}^m c_{k}\langle
\nabla^{k}u,\nabla^{k}v\rangle_x$ and $c_0,\dots,c_m$ are the
appropriate constants guaranteed by (\ref{invariance}).  The latter
expression allows us to extend naturally the native space inner
product to measurable subsets $\Omega$ of $\sphere$.  Namely,
$$\langle u,v\rangle_{\Omega, k_m} :=\int_{\Omega} \beta(u,v)_x \dif \mu(x).$$
This has the desirable property of set additivity: for \ sets $A$ and
$B$ with $\mu(A\cap B) = 0$, we have $\langle u,v\rangle_{A\cup B,
  k_m} = \langle u,v\rangle_{A, k_m}+\langle u,v\rangle_{ B, k_m}.$
Unfortunately, since some of the coefficients $c_k$ may be negative,
$\beta(u,u)$ and $\langle u,u\rangle_{\Omega, k_m}$ may assume
negative values for some $u$: in other words, the bilinear form
$(u,v)\mapsto \langle u,v\rangle_{\Omega, k_m}$ is only an {\em
  indefinite} inner product.

{\bf A Cauchy-Schwarz type inequality.} 
When restricted to the cone of functions in $W_2^m(\Omega)$ having a
sufficiently dense set of zeros, the quadratic form $ \langle
u,u\rangle_{\Omega, k_m}$ is positive definite. We now briefly discuss
this.

When $\Omega$ has Lipschitz boundary and $u$ has many zeros, we can relate the 
quadratic form  
$\ns{u}_{\Omega,k_m}^2:=\langle u,u\rangle_{\Omega, k_m} $ 
to a Sobolev norm $\|u\|_{W_2^m(\Omega)}^2$. Arguing as in 
\cite[(4.2)]{HNW2}, 
we see that
\begin{eqnarray*}
c_m |u|_{W_2^m(\Omega)}^2
 &-&
\bigl(\max_{j\le m-1}{|c_j|}\bigr)\|u\|_{W_2^{m-1}(\Omega)}^2\\
&\le&
 \int_{\Omega}  \beta(u,u)_x \dif \mu (x) \\
&\le &
\bigl(\max_{j\le m}{|c_j|}\bigr) \|u\|_{W_2^m(\Omega)}^2.
\end{eqnarray*}
If $u|_{\Xi} = 0$ on a set $\Xi$ with $h(\Xi,\Omega) \le h_0$ with
$h_0$ determined only by the boundary of $\Omega$ (specifically the
radius and aperture of an interior cone condition satisfied by
$\partial \Omega$), Theorem A.11 of \cite{HNW2} guarantees that
$\|u\|_{W_2^{m-1}(\Omega)}^2 \le Ch^2 |u|^2_{W_2^m(\Omega)}$ with $C$
depending only on the order $m$ and the roughness of the boundary (in
this case, depending only on the aperture of the interior cone
condition). Thus, by choosing
$h\le h^{*}$, where $h^{*}$ satisfies the two conditions 
\begin{equation}\label{hstar}
h^*\le h_0\quad \text{and}
 \quad
 C (h^*)^2\times \bigl(\max_{j\le m}{|c_j|}\bigr) \le \frac{|c_m|}{2},
 \end{equation}  
 we have 
$$
\frac{c_m}{2}\|u\|_{W_2^m(\Omega)}^2
\le
\ns{u}_{\Omega,k_m}^2
\le 
 \left(\max_{j\le m}|c_j|\right)
 \|u\|_{W_2^m(\Omega)}^2. 
$$
  
The threshold value $h^*$ depends on the coefficients $c_j$ as well as
the radius $R_{\Omega}$ and aperture $\phi_{\Omega}$ of the cone
condition for $\Omega$.  When $\Omega$ is an annulus of sufficiently
small inner radius, the cone parameters can be replaced by a single
global constant, and $h_*$ can be taken to depend only on $c_0,\dots,
c_m$. In other words, only on $k_m$ -- cf. \cite[Corollary
A.16]{HNW2}.

A direct consequence of this is positive definiteness for such
functions, $\ns{u}_{\Omega,k_m}\ge 0$ with equality only if
$u|_{\Omega}=0$. From this, we have a version of the Cauchy-Schwarz
inequality: if $u$ and $v$ share a set of zeros $Z$ (i.e., $u|_Z =
v|_Z = \{0\}$) that is sufficiently dense in $\Omega$, then
\begin{equation}\label{C-S} \left|\langle u,v\rangle_{\Omega, k_m}
\right| \le \ns{u}_{\Omega,k_m}\ns{v}_{\Omega,k_m}
\end{equation} follows (sufficient density means that $h(Z,\Omega)<
h^*$ as above).

{\bf Decay of Lagrange functions.}
\cite[Lemma 5.1]{HNW2} guarantees that the Lagrange function $\chi_{\xi}$
satisfies the {\em bulk chasing} estimate: there is a fixed constant
$0\le \epsilon<1$ so that for radii $r$ 
the estimate 
$$ \|\chi_{\xi}\|_{W_2^m(\comp(\xi,r))} \le 
\epsilon \|\chi_{\xi}\|_{W_2^m(\comp(\xi,r -\frac{h}{4h_0}))} $$ 
holds.
In other words, a fraction (roughly $1-\epsilon$) of the bulk of the tail 
$\|\chi_{\xi}\|_{W_2^m(\comp(\xi,r))}$ is to be found in the
annulus $ \b(\xi,r)\setminus \b(\xi,r -\frac{h}{4h_0})$ of width
$\frac{h}{4h_0} \propto h$ (with a constant of proportionality
$\frac{1}{4h_0}$ that depends only on  $m$). 
For $r>0$, it is possible to iterate this $n$
times, provided  $n\frac{h}{4h_0}\le r$.  It follows that there is $\nu = -
4h_0 \log \epsilon>0$ so that
\begin{equation*}
\| \chi_{\xi}\|_{W_2^m(\comp(\xi,r))}
\le \epsilon^n \|\chi_{\xi} \|_{W_2^m(\sphere)}
\le Ce^{-\nu r/h}\|\chi_{\xi} \|_{W_2^m(\sphere)}.
\end{equation*} 
By \cite[(5.1)]{HNW2} \footnote{This is simply a comparison of $\chi_{\xi}$ to a smooth
``bump'' $\phi_{\xi}$ of radius $q$ -- also an interpolant to the
delta data $(\delta_{\xi})$, but worse in the sense that
$\ns{\chi_{\xi}}_{k_m}\le \ns{\phi}_{k_m}$ -- this idea is repeated in the proof of
Theorem \ref{main}.} we have
\begin{equation}
\label{Lag_decay}
\| \chi_{\xi}\|_{W_2^m(\comp(\xi,r))}
\le
C q^{1-m} e^{-\nu \frac{r}{h}}.
\end{equation}
This leads us to our main result.

\begin{theorem}\label{main} 
   Let $\rho>0$ be a fixed mesh ratio.
  There exist constants $h^*$, $\nu$ and $C$ depending only on $m$ and $\rho$
  so that if $h\le h^*$, then the Lagrange
  function $\chi_{\zeta} = \sum_{\xi\in\Xi} A_{\zeta,\xi}
  k_m(\cdot,\xi) + p_{\zeta}\in S_m(\Xi)$ has these properties:
\begin{align}
  |\chi_{\xi}| &\le
  C \exp\left(-\nu\frac{\d(x,\xi)}{h}\right). \label{lagrange_decay}\\
  |A_{\zeta,\xi}|
  & \le C q^{2-2m} \exp{\left(-\nu
      \frac{\d(\xi,\zeta)}{h}\right)}. \label{Coeff}
  \\[3pt]
  c_1
  q^{2/p} \|\bfa\|_{\ell_p(\Xi)} &\le \big\|\sum_{\xi\in\Xi} a_{\xi}
  \chi_{\xi}\big\|_{L_p(\sphere)} \le c_2 q^{2/p} \|\bfa\|_{\ell_p(\Xi)}.
  \quad \label{p_stability}
\end{align}
\end{theorem}

\begin{proof}
  The bounds (\ref{lagrange_decay}) and (\ref{p_stability}) are given
  in \cite[Theorems~5.3 \& 5.7]{HNW2}. Only (\ref{Coeff}) requires
  proof. By Proposition \ref{Lagrange_Coeffs_Formula} and set
  additivity, we have that
\[
A_{\zeta,\xi} = \langle \chi_{\xi},\chi_{\zeta} \rangle_{ k_m} =
\langle \chi_{\xi},\chi_{\zeta} \rangle_{ \Omega_{\zeta} ,k_m} +
\langle \chi_{\xi},\chi_{\zeta} \rangle_{ \Omega_{\xi}, k_m},
\]
where we employ the 
hemispheres: $\Omega_{\zeta} =
\left\{ \alpha \in \sphere\mid \d(\alpha,\zeta) <
  \d(\alpha,\xi)\right\},$ $\Omega_{\xi} = \left\{ \alpha \in
  \sphere\mid \d(\alpha,\xi) < \d(\alpha,\zeta)\right\}. $ Modulo
a set of measure zero, $\Omega_{\xi}=\sphere\setminus\Omega_{\zeta} 
$.

We apply the Cauchy--Schwarz type inequality (\ref{C-S}) to obtain
\begin{eqnarray*} |A_{\zeta,\xi} | &\le&
\ns{\chi_{\zeta}}_{ \Omega_{\zeta} ,k_m}
 \ns{\chi_{\xi}}_{ \Omega_{\zeta} ,k_m}+
\ns{\chi_{\zeta}}_{ \Omega_{\xi}, k_m} 
\ns{\chi_{\xi}}_{ \Omega_{\xi}, k_m}\\
&\le&
\sqrt{\max_{j\le m}|c_j|}
\left(\|\chi_{\zeta}\|_{ W_2^m(\Omega_{\zeta} )}
 \ns{\chi_{\xi}}_{ \Omega_{\zeta} ,k_m }+
\ns{\chi_{\zeta}}_{ \Omega_{\xi}, k_m} 
\|\chi_{\xi}\|_{ W_2^m(\Omega_{\xi})}\right)
\end{eqnarray*}

Since $ \Omega_{\zeta} \subset \b^c(\zeta,r) := \sphere\setminus
\b\left(\zeta,\frac{1}{2} \d(\xi,\zeta)\right) $ 
and 
$\Omega_{\xi} \subset \b^c(\xi,r)$, 
we can again employ set additivity and
positive definiteness (this time $ \ns{\chi_{\xi}}_{ \Omega_{\zeta} ,k_m }\le \ns{\chi_{\xi}}_{ \sphere ,k_m } $, 
which follows from
the fact that $\sphere = \Omega_{\zeta} \cup \overline{\Omega_{\xi}}$ and
that $\chi_{\xi}$ vanishes to high order in $\Omega_{\xi}$ -- the same
holds for $\chi_{\zeta}$) to obtain
$$ |A_{\zeta,\xi} | \le\sqrt{\max_{j\le m}|c_j|}\left(\|\chi_{\zeta}\|_{ W_2^m(\b^c(\zeta,r) ) } \ns{\chi_{\xi}}_{ k_m}+
\ns{\chi_{\zeta}}_{  k_m} \|\chi_{\xi}\|_{W_2^m( \b^c(\xi,r)   )}\right).$$

The full energy of the Lagrange function can be bounded by comparing
it to the energy of a bump function -- for $\chi_{\xi}$ this is
$\phi_{\xi}$, which can be defined 
by using a smooth cutoff function $\sigma$. In spherical coordinates
(colatitude, longitude) around $\xi$,
$\phi_{\xi}(\theta,\varphi) = \sigma(\theta/q).$ 
This is done in
\cite[(5.1)]{HNW2} and we have that $\ns{\chi_{\xi}}_{ k_m}$
and $\ns{\chi_{\zeta}}_{ k_m}$ are bounded by $C q^{1-m}$.

On the other hand, we can employ (\ref{Lag_decay}) to treat 
$\|\chi_{\zeta}\|_{ W_2^m(\b^c(\zeta,r) ) }$ and 
$ \|\chi_{\xi}\|_{ W_2^m(\b^c(\zeta,r) ) }$, which gives
\[
\| \chi_{\xi} \|_{ W_2^m(\b^c\left(\zeta,r\right)) }, \| \chi_{\zeta}
\|_{ W_2^m(\b^c\left(\zeta,r\right)) } \le C q^{1-m} e^{-\nu
  \frac{r}{h}} = C q^{1-m} e^{-\nu \frac{\d(\xi,\zeta)}{2h}}.
\]
The bound (\ref{Coeff}) follows immediately from this.
\end{proof}

\begin{remark}\label{general_manifold_main} 
{\rm Because the proof doesn't really depend on $\sphere$, a nearly
identical proof works for any of the kernels with exponentially
decaying Lagrange functions considered in
\cite{HNW,HNW2}. Specifically, we have this:} Theorem~\ref{main} holds
for compact, 2-point homogeneous spaces with polyharmonic kernels
satisfying $\L_m \perp \Pi$ (cf.~\cite{HNW2}) and for any compact,
$C^\infty$ Riemannian manifold, with the kernels being the Sobolev
splines given in \cite{HNW}.
\end{remark}
\section{Truncating the Lagrange basis}\label{S:better_basis}

We now discuss truncating the kernel expansion Lagrange function
$\chi_\xi = \sum_{\zeta\Xi} A_{\xi,\zeta}k_m(\cdot,\zeta)+p_\xi\in
S_m(\Xi)$, replacing it with an expansion of the form
\begin{equation}
\label{approx_lagrange}
\widetilde \chi_\xi = \sum_{\zeta\in\Upsilon(\xi)} \widetilde A_{\xi,\zeta}
k_m(\cdot,\zeta)+p_\xi \in S_{m}(\Xi),
\end{equation}
where $\Upsilon(\xi)\subset \Xi$ is a set of centers contained in a
ball $B(\xi,r(h))$ centered at $\xi$, where $r(h)$ and the $\widetilde
A_{\xi,\zeta}$'s will be determined by $A_{\xi,\zeta}$, with $\zeta\in
\Upsilon(\xi)$.  We also assume that $\xi\in \Upsilon(\xi)$.  Finally,
to avoid notational clutter, we will simply use $\Upsilon$ rather than
$\Upsilon(\xi)$.

Our goal is to show that if $\chi_\xi$ satisfies the properties
(\ref{lagrange_decay}), (\ref{Coeff}), and (\ref{p_stability}), then
we may take $r(h)=Kh|\log(h)|$, with $K=K(m)>0$, while maintaining
algebraic decay in $h$ of the error $\|\widetilde \chi_\xi -
\chi_\xi\|_\infty $. For this choice of $r(h)$, a simple volume
estimate (given at the end of Section \ref{loc_lag_bases}) shows that
the number of terms required for $\widetilde \chi_\xi$ is just
$\calo((\log N)^2) \ll N$, far fewer than the $N$ needed for
$\chi_\xi$.


Simply truncating at a fixed radius $r(h)=Kh|\log(h)|$ is not
suitable, however, because the truncated function $\widetilde
\chi_\xi$ will no longer be in the space $S_{m}(\Xi)$ (and thus
$\{\widetilde\chi_\xi\}$ will not act as a basis). To treat this, we
must slightly realign coefficients to satisfy the moment conditions.

A remark before proceeding with the analysis: Finding $\widetilde
\chi_\xi$ in the way described below requires knowing the expansion
for $\chi_\xi$ and carrying out the truncation above. This is
expensive, although it does have utility in terms of speeding up
evaluations for interpolation when the same set of centers is to be
used repeatedly. The main point is that we now know roughly how many
basis elements are required to obtain a good approximation to
$\chi_\xi$. The question of producing a good basis efficiently is left
to the next section.

\subsection{Constraint conditions on the
  coefficients}\label{realigning}
We would like $\widetilde \chi_{\xi}$ to be in the space $S_{m}(\Xi)$,
and so the $\widetilde A_{\xi,\zeta}$'s have to satisfy the
constraints in the system (\ref{Collocation_System}):
\begin{equation}
\label{tilde_A_constraints}
\sum_{\zeta\in\Upsilon } \widetilde A_{\xi,\zeta} \overline{\phi_j}(\zeta) = 0, \ j\in \J:=(1,\dots,m^2),
\end{equation}
where $\{\phi_j\}_{j=1}^{m^2}$ is an orthonormal basis for $\Pi_{m-1}$. Since the original $\chi_\xi$'s are in $S_{m}(\Xi)$, the
$A_{\xi,\zeta}$'s in their expansions satisfy the constraint equations
in (\ref{Collocation_System}). Splitting these equations into sums
over $\Upsilon $ and its complement in $\Xi$ and manipulating the
result, we see that
\begin{equation}
\label{A_constraints}
 \sum_{\zeta \in \Upsilon } A_{\xi,\zeta} \overline{\phi_j}(\zeta) +\sigma_j, \quad \text{where } 
 \sigma_j :=   \sum_{\zeta \not\in \Upsilon } A_{\xi,\zeta} \overline{\phi_j}(\zeta), \ j\in \J .
\end{equation}
The way that we will relate the two sets of coefficients is to define
the vector $(\widetilde A_{\xi,\zeta})_{\zeta\in \Upsilon }$ to be the
orthogonal projection of $(A_{\xi,\zeta})_{\zeta\in \Upsilon }$ onto
the constraint space, which is the orthogonal complement of $\spam\{
\left.\phi_j\right|_{\Upsilon}
, j\le m^2\}$, in the usual inner product for $\ell_2(\Upsilon)$.  The
equations below then follow:
\begin{equation}
\label{least_squares_A}
\begin{aligned}
(\widetilde A_{\xi,\zeta})_{\zeta\in \Upsilon } - (A_{\xi,\zeta})_{\zeta\in \Upsilon }
&= \sum_{j\in \J} \tau_j \phi _j |_{\Upsilon } \in \spam\{(\phi_j(\zeta))_{\zeta\in \Upsilon }, j\le m^2\} \\
\|(\widetilde A_{\xi,\zeta})_{\zeta\in \Upsilon } - (A_{\xi,\zeta})_{\zeta\in \Upsilon }\|_{\ell_2(\Upsilon)}^2 & = \tau^* G_{\Upsilon }\tau , \ \   [G_{\Upsilon }]_{k,j} := \textstyle{\sum_{\zeta\in \Upsilon }} \overline{\phi_k}(\zeta)\phi_j(\zeta),
\end{aligned}
\end{equation}
where $\tau$ is a column vector having the $\tau_j$'s as entries. Let
$\sigma$ be a column vector with the $\sigma_j$'s as entries. From the
first equation above together with equations
(\ref{tilde_A_constraints}) and (\ref{A_constraints}), $\tau$ and
$\sigma$ are related by $\sigma=G_{\Upsilon }\tau$. If we make the
rather mild assumption that $\Upsilon $ is unisolvent for the space
$\Pi_{m-1}$, then we can invert $G_{\Upsilon }$: $\tau=G_{\Upsilon
}^{-1}\sigma$, thereby obtaining $\tau^* G_{\Upsilon }\tau = \sigma^*
G_{\Upsilon }^{-1}\sigma$.  Using this in (\ref{least_squares_A}) and
applying Schwarz's inequality,
we obtain this bound:
\begin{equation}
\label{est_on_inner_part}
\big\|  \textstyle{\sum_{\zeta\in\Upsilon }} (\widetilde A_{\xi,\zeta} - A_{\xi,\zeta})
k_{m}(\cdot,\zeta)\big\|_\infty \le \sqrt{ \#\Upsilon \, \| G_{\Upsilon }^{-1}\|_2} \, \|k_{m}\|_\infty \|\sigma\|_2,
\end{equation}
which we will make use of to establish the estimates below.

\begin{proposition} \label{chi_tilde_properties}
Assume that $\Upsilon $ is unisolvent for $\Pi_{m-1}$ and that 
$\| G_{\Upsilon }^{-1}\|_2=\calo(|\log h|^{-2}\,h^{-2\mu})$, for some
$\mu\ge 0$.  If we take $r(h)=K h |\log(h)|$, where $K$ is chosen so
that $J:= K\nu - 2m-\mu>0$ then for $h$ sufficiently small,
\begin{gather}
\|\widetilde \chi_\xi - \chi_\xi\|_\infty \le C h^{J} \label{error_tilde_chi}\\
|\widetilde \chi_{\xi}(x)| \le C\big(1+\d(x,\xi)/h\big)^{-J} \label{tilde_chi_bnd}
\end{gather}
Furthermore, when $J>2$, the set $\{\widetilde \chi_\xi\}$ is $L_p$ stable: there are $C_1,C_2>0$ for which
\begin{equation}
\label{p_stable_J}
C_1 q^{2/p} \|\bfa\|_{\ell_p(\Xi)}
\le 
\big\|\textstyle{\sum_{\xi\in\Xi}} a_{\xi} \widetilde \chi_{\xi}\big\|_{L_p(\sphere)}
\le 
C_2 q^{2/p} \|\bfa\|_{\ell_p(\Xi)}.
\end{equation}
\end{proposition}

\begin{proof} From (\ref{Coeff}) and $N\le 4\pi/\mathrm{vol}(B(\xi,q))\le Cq^{-2}$, we have that
\begin{equation}
\label{A_l1_estimate}
\sum_{\zeta \not\in \Upsilon} |A_{\xi,\zeta}| =\calo\big( Nq^{2-2m} \exp(-\nu r(h)/h)\big) \le C h^{K\nu - 2m}.
\end{equation}
Applying it to the  $\sigma_j$'s defined in (\ref{A_constraints}) results in $
\|\sigma\|_2 \le C  h^{K\nu - 2m} $.  Using this in connection with (\ref{est_on_inner_part}), $\| G_{\Upsilon }^{-1}\|_2=\calo(|\log h|^{-2}\,h^{-2\mu})$, (\ref{A_l1_estimate}) and 
\[
\widetilde \chi_\xi-\chi_\xi =  \sum_{\zeta\in\Upsilon} (\widetilde A_{\xi,\zeta} - A_{\xi,\zeta})
k_m(\cdot,\zeta)- \textstyle{\sum_{\zeta \not\in \Upsilon}} A_{\xi,\zeta} k_m(\cdot,\zeta),
\]
yields (\ref{error_tilde_chi}). Next, from (\ref{lagrange_decay}) we have
\begin{eqnarray*}
|\chi_{\xi}| &\le& C \exp\left(-\nu\frac{\d(x,\xi)}{h}\right)\\
& \le& C \exp\left(-K\nu\frac{\d(x,\xi)}{Kh}\right)\\
& \le& 
C\bigg(1+\frac{\d(x,\xi)}{Kh}\bigg)^{-K\nu}.
\end{eqnarray*}
Combining this with (\ref{error_tilde_chi}), using $J=K\nu -2m-\mu >0$ and manipulating, we arrive at (\ref{tilde_chi_bnd}).

It remains to demonstrate the $L_p$ stability of $(\tilde{\chi}_{\xi})$ for $1\le p\le \infty$. When $p=1$, we 
consider  a sequence $\bfa = (a_{\xi})_{\xi\in\Xi}\in \ell_{1}(\Xi)$. Let $s := \sum a_{\xi} \chi_\xi$ and $\tilde{s} := \sum a_{\xi} \widetilde{\chi}_{\xi}$. From H{\" o}lder's inequality and (\ref{error_tilde_chi}), we have 
$\|\tilde{s}  - s \|_{L_1(\sphere)} \le C \|\bfa\|_{\ell_1(\Xi)} h^{J}$ and
\[
\|\tilde{s}  - s \|_{L_\infty(\sphere)} 
\le 
C \|\bfa\|_{\ell_\infty(\Xi)}
\underbrace{ \textstyle{\sum_{\xi \in \Xi}}|\widetilde \chi_\xi(x) - \chi_\xi(x)|}_{\le N\max_\xi \| \widetilde \chi_\xi - \chi_\xi\|_{L_\infty(\sphere)}}
\le C\|\bfa\|_{\ell_\infty(\Xi)}  h^{J}q^{-2}.
\]
Interpolating between these two inequalities -- i.e.,  interpolating
the finite rank operator $\bfa \mapsto (s - \tilde{s})$ --
gives
\[
\begin{aligned}
\|s - \tilde{s}\|_{L_p(\sphere)} &\le C   h^{J} q^{ -2(1-1/p) }\|\bfa\|_{\ell_p(\Xi)} \\
&\le Ch^{J-2}q^{2/p}\|\bfa\|_{\ell_p(\Xi)}. \quad (q^{-2}\sim h^{-2}).
\end{aligned}
\]
After some manipulation,  this bound and (\ref{p_stability}) imply that
\[
c_1q^{2/p}\|\bfa\|_{\ell_p(\Xi)} (1- C h^{J-2} )
\le \|\tilde{s}\|_{L_p(\sphere)}
\le
c_2q^{2/p}\|a\|_{\ell_p(\Xi)}(1+ C h^{J-2}).
\]
Choosing $h$ so that $C h^{J-2} \le 1/2$ and letting $C_1=c_1/2$ and
$C_2=3c_2/2$, we obtain (\ref{p_stable_J}).
\end{proof}

\begin{remark}\label{general_manifold_loc}
{\rm When there are no constraint conditions on the coefficients, this
result holds for any of the strictly positive definite kernels
mentioned in Remark~\ref{general_manifold_main}}. In particular it
holds for Sobolev splines on a compact $C^\infty$ Riemannian manifold.
\end{remark}

%
%

\subsection{Norm of the inverse Gram matrix}\label{spheres_proj_lagrange}

We now demonstrate that the conditions on $G_{\Upsilon }^{-1}$ in
Proposition \ref{chi_tilde_properties} are automatically satisfied. We
will state and prove the results below for caps on $\sph^d$, rather
than just $\sphere$. Also, It is more convenient to use with $\Pi_L$
rather than $\Pi_{m-1}$, because $m$ is notationally tied to the
polyharmonic kernels $k_m$ as well as the spherical harmonics on
$\sphere$. That said, we begin with the lemma below.

\begin{lemma} \label{quad_cap_lem} Suppose that $S_r :=B(\xi,r)\subset
  \sph^d$ is a cap of fixed radius $r<\pi$, and that $\scrc \subset
  S_r$ is finite and has mesh norm $h_\scrc :=h_{S_r,\scrc}$. In
  addition, let $L\ge 0$ be a fixed integer and take $\Pi_L$ to be
  the space of all spherical harmonics of degree at most $L$. Then,
  there exists a constant $c_0:=c_0(d,L)>0$ such that when $h_\scrc
  \le c_0r$ we have
\begin{equation}
\label{quad_cap_ineq}
\sum_{\zeta\in \scrc} |\varphi(\zeta)|^2 \ge \mu(S_r)^{-1}\int_{S_r} |
\varphi(x)|^2 d\mu(x), \ \text{for all }\varphi\in \Pi_L.
\end{equation}
Moreover, the set $\scrc$ is unisolvent for $\Pi_L$. Finally, for
every basis for $\Pi_L$ the corresponding Gram matrices $G_\scrc$ and
$G_{S_r}$, relative to the inner products on $\ell^2(\scrc)$ and
$S_r$, respectively, satisfy
\begin{equation}
\label{gram_comparison}
\|G_\scrc^{-1}\|_2 \le \mu(S_r) \|G_{S_r}^{-1}\|_2.  
\end{equation}
\end{lemma}

\begin{proof}
  Since since $\varphi(x)$ and $\overline{\varphi}(x)$ are spherical
  harmonics in $\Pi_L$, their product is a spherical harmonic of
  degree at most $2L$.  Thus, applying the nonnegative-weight
  quadrature formula in \cite[Theorem~2.1]{mhaskar-2004-1} to
  spherical harmonics of order $2L$ yields
\[
\sum_{\zeta\in \scrc} w_\zeta |\varphi(\zeta)|^2 = \int_{S_r} |\varphi(x)|^2 d\mu(x), 
\]
Since $0\le w_\zeta \le \sum_{\zeta\in \scrc} w_\zeta = \mu(S_r)$, we
have the inequality $\sum_{\zeta\in \scrc} w_\zeta |\varphi(\zeta)|^2\le \mu(S_r)
\sum_{\zeta\in \calc} |\varphi(\zeta)|^2$. The inequality
(\ref{quad_cap_ineq}) follows immediately from the quadrature
formula. To prove that $\calc$ is unisolvent, suppose that $\varphi\in
\Pi_L$ vanishes on $\calc$. By (\ref{quad_cap_ineq}), we have that
$\int_{S_r} |\varphi(x)|^2 d\mu(x)=0$.  Since $\varphi$ is in $\Pi_L$,
it is a polynomial in sines and cosines of the angles used in the
standard parameterization of $\sph^d$, with $\xi$ being the ``north''
pole. As a consequence, it is continuous on $S_r$ and, because
$\int_{S_r} |\varphi(x)|^2 d\mu(x)=0$, it is identically $0$ on
$S_r$. Finally, as a function of the angular variables in the complex
plane, it is analytic, entire in fact, and can be expanded in a power
series in these variables. The fact that it vanishes identically for
\emph{real} values of the angular variables is enough to show that the
coefficients in the series are all zero. Hence, $\varphi \equiv 0$ on
$\sph^d$ and $\scrc$ is unisolvent for $\Pi_L$. To establish
(\ref{gram_comparison}), note that (\ref{quad_cap_ineq}) implies that
$G_\scrc - \mu(S_r)^{-1} G_{S_r}$ is positive semi definite. From the
Courant-Fischer theorem, the lowest eigenvalue of $G_\scrc$ is greater
than that of $ \mu(S_r)^{-1} G_{S_r}$. This inequality then yields
(\ref{gram_comparison}), since these eigenvalues are
$\|G_\scrc^{-1}\|_2^{-1}$ and $\mu(S_r) \|G_{S_r}^{-1}\|_2^{-1}$,
respectively.
\end{proof}

We now need to compute the Gram matrix for the \emph{canonical basis}
of $\Pi_L$. This basis is described in \cite[Chapter IX, \S
3.6]{vilenkin_book_1968} and consists of spherical harmonics. Let
$\ell,k_1,\ldots,k_{d-1}$ be integers satisfying $\ell\ge k_1\ge
k_2\ge\cdots \ge k_{d-1}\ge 0$, and take $K:=(k_1,\ldots,\pm
k_{d-1})$. A spherical harmonic of degree $\ell$
\cite[p.~466]{vilenkin_book_1968}) will be denoted by
$Y^\ell_K(\theta_1,\ldots,\theta_{d})$. The angles are the usual ones
from spherical coordinates in $\RR^{d+1}$
(cf. \cite[p.~435]{vilenkin_book_1968}). The basis for $\Pi_L$ is then
the set of all $Y^\ell_K$, $0\le \ell\le L$. The entries in the Gram
matrix are $[G_{S_r}]_{(\ell, K),(\ell',K')}=\langle
Y^\ell_K,Y^{\ell'}_{K'}\rangle_{S_r}$. Following the argument in
\cite[Chapter IX, \S 3.6]{vilenkin_book_1968}, one may show that
\begin{multline}
\label{gram_entries}
\langle Y^\ell_K,Y^{\ell'}_{K'}\rangle_{S_r} \\
=
B_{\ell,K}B_{\ell',K}\delta_{K,K'} 
\int_0^r C^{\frac{d-1}{2}+k_1}_{\ell-k_1}(\cos \theta)
C^{\frac{d-1}{2}+k_1}_
{\ell'-k_1}(\cos \theta)\sin^{2k_1+d-1} \theta d\theta, 
\end{multline}
where $C^s_n(t)$ is the Gegenbauer polynomial of degree $n$ and type
$s$, and $B_{\ell,k_1}$ is a normalization factor. In the case where
$d=2$ and $L=1$, $G_{S_r}$ is $4\times 4$ and has six non-zero
entries,
\begin{equation*}
\begin{array}{rrcl}
 & G_{(0,0),(0,0)} &=& \frac12 (1-\cos r), \\
&G_{(0,0),(1,0)} =
  G_{(1,0),(0,0)}&=&\frac{\sqrt{3}}{4}(1- \cos r)(1+\cos r), \\
  &G_{(1,0),(1,0)}&=&\frac12 (1-\cos r)(1+\cos r +\cos^2r),\\
   & G_{(1,\pm 1),(1,\pm 1)}&=&\frac14 (1-\cos r)^2(2 + \cos r ) .
\end{array}
\end{equation*}
Since $\mu(S_r) = 2\pi (1-\cos r)$, the formulas for the entries above
imply that $G_{S_r}/\mu(S_r)$ is a polynomial in $\cos r$. In fact, a
straightforward calculation shows that the minimum eigenvalue of this
matrix is $r^4/(256\pi) + \calo(r^6)$. Lemma~\ref{quad_cap_lem} then
implies that $\|G_\calc^{-1}\|_2 \le \mu(S_r)\|
G_{S_r}^{-1}\|_2=256\pi r^{-4}+\calo(r^{-2})$. A less precise, but
similar result, holds in the general case.

\begin{lemma}\label{norm_inverse_lower_bnd}
  Under the assumptions of Lemma~\ref{quad_cap_lem}, for general $L\ge
  0$, $d\ge 2$ and $r$ sufficiently small, there is an integer
  $\iota=\iota (L,d)\ge L$ and a constant $C=C(L,d)>0$ such that
  $\|G_\calc^{-1}\|_2 \le Cr^{-2\iota}$. For $L=1$ and $d=2$, we may
  take $\iota = 2$.
\end{lemma}

\begin{proof}
  From the expression in (\ref{gram_entries}) for the entries in
  $G_{S_r}$, we see that each of them is entire in $r$ and has a zero
  of order $d$ or greater at $r=0$. In addition, $\mu(S_r)$ is also
  entire in $r$ and has a zero of order $d$. It follows that the
  matrix $\widetilde G(r) = G_{S_r}/\mu(S_r)$ is entire, even in
  $r$. and for real $r$, it is real, self adjoint and positive semi
  definite. (In fact, for $d$ even, it is a polynomial in $\cos r$.)
  In addition, the $2\times 2$ block in $\widetilde G(0)$
  corresponding to $k_1=0$, $\ell=0,1$, is rank 1 and therefore has 0
  as an eigenvalue; consequently, $\widetilde G(0)$ also has 0 as an
  eigenvalue -- it's lowest, in fact. As Rellich \cite[pg.\
  91]{rellich-book-1969} shows, the eigenvalues of $\widetilde G(r)$
  are analytic functions of $r$. For $r>0$, these eigenvalues are
  proportional to those of the Gram matrix $G_{S_r}$ and therefore
  must be positive. None of these eigenvalues are identically $0$. In
  particular, the eigenvalues splitting off from the $0$ eigenvalue of
  $\widetilde G(0)$ are not identically $0$. As functions of $r$ they
  thus have a zero of finite order at $r=0$; the order is an even
  integer because $\widetilde G(r)$ is even in $r$. The smallest
  eigenvalue then behaves like $\lambda_{min}(r) = r^{2\iota}(a_0+
  \calo(r^2))$, where $a_0>0$, $\iota>0$ is an integer, and $r$ is
  sufficiently small. Furthermore, from (\ref{gram_entries}) we see
  that the diagonal entry, with $\ell=\ell'=k_1=L$, is $\calo
  (r^{2L})$.  since this bounds the minimum eigenvalue from above, we
  must have $2\iota \ge 2L$, so $\iota \ge L$. The result then follows
  from Lemma~\ref{quad_cap_lem} and the observation that
  $\mu(S_r)\|G_{S_r}^{-1}\|_2 = \lambda_{min}^{-1}$. The calculation
  for $L=1$ and $d=2$ was done above.
\end{proof}

\subsection{Local Lagrange Bases} \label{loc_lag_bases}

We now turn to the \emph{local} Lagrange basis. Recall that the
function $\achi_\xi \in S_{m}(\Xi)$, with the kernel representation
\begin{equation}\label{local_lagrange}
\check \chi_\xi = \sum_{\zeta\in\Upsilon} \check A_{\xi,\zeta}
k_m(\cdot,\zeta)+\sum_{j=1}^{m^2}\check b_j \phi_j \in S_m(\Xi),
\end{equation}
is a \emph{local} Lagrange function centered at $\xi$ if it satisfies
$\check \chi_\xi|_\Upsilon = \bfe_\xi$, where $\bfe_\xi(\zeta) =
\delta_{\xi,\zeta}$; that is, $\bfe_\xi$ is the vector
$(1,0,\ldots,0)^T$. Since $\check \chi_\xi \in S_m(\Xi)$, the vector
$\check A_\xi = (\check A_{\xi,\zeta})_{\zeta\in \Upsilon}$ is in the
constraint space. This vector and the coefficients $\check b_j$ then
satisfy $\check \chi_\xi|_{\Upsilon} = \bfe_\xi$, Of course, the
(full) Lagrange function $\chi_\xi=\sum_{\zeta\in\Xi} A_{\xi,\zeta}
\kappa(\cdot,\zeta)+\sum_{j=1}^{m^2} b_j \phi_j$ restricted to
$\Upsilon$ also satisfies $\chi_\xi|_{\Upsilon} =
\bfe_\xi$. Consequently, $\mathfrak D_\xi :=\check \chi_\xi -\chi_\xi$
satisfies $\mathfrak D_\xi|_{\Upsilon} = \bf0$. We can rewrite this
difference as $\mathfrak D_\xi = \check \chi_\xi -\widetilde \chi_\xi
+ \widetilde \chi_\xi -\chi_\xi$ (with $\widetilde \chi_\xi $ the
truncated basis function introduced in the last section).  It follows
that
\[
\mathfrak D_\xi = \underbrace{\sum_{\zeta \in \Upsilon}
  (\underbrace{\check A_\zeta - \widetilde
    A_\zeta}_{\displaystyle{\alpha_\zeta}} ) \kappa(\cdot,\zeta) +
  \sum_{j=1}^{\# \J}(\underbrace{\check b_j -
    b_j}_{\displaystyle{\beta_j}} ) \phi_j}_{\check \chi_\xi -
  \widetilde \chi_\xi} + \widetilde \chi_\xi - \chi_\xi .
\]
Evaluating this on $\Upsilon$ then gives the system
$\K_{\Upsilon}\alpha + \Phi \beta +(\widetilde \chi_\xi -
\chi_\xi)|_{\Upsilon} = \bf0$.
By linearity, it is clear that $\Phi^\ast \alpha=\bf0$. Finally,
letting $\bfy= (\chi_\xi - \widetilde \chi_\xi)|_{\Upsilon}$, we
arrive at the system,
\begin{equation}
\label{local_lag__matrix}
\begin{pmatrix}
\K_{\Upsilon} & \Phi\\
\Phi^* &\bf0
\end{pmatrix}
\begin{pmatrix}
\alpha \\ \beta\end{pmatrix}
=
\begin{pmatrix}
\bfy\\\bfzero\end{pmatrix}.
\end{equation}
Proposition~\ref{inverse_norm_est} applies to
(\ref{local_lag__matrix}), with $\Xi$ replaced by $\Upsilon$; thus,
noting that $\|\bfy\|_{\ell_\infty(\Upsilon)}\le \| \widetilde
\chi_\xi - \chi_\xi\|_\infty$, and writing $\J = (1,\dots,m^2)$, we
see that
\begin{gather*}
  \|\alpha\|_{\ell_2(\Upsilon)} \le \vartheta^{-1}\sqrt{\# \Upsilon}
  \|\widetilde \chi_\xi -
  \chi_\xi\|_\infty \\
  \|\beta\|_{\ell_2(\J)}\le 2\|k_m\|_\infty \| G_\Upsilon^{-1}\|^{1/2}
  \vartheta^{-1}(\# \Upsilon)^{3/2} \|\widetilde \chi_\xi -
  \chi_\xi\|_\infty
\end{gather*}
From this we obtain these inequalities:
\begin{align*}
  \|\check \chi_\xi - \widetilde \chi_\xi\|_\infty &\le \|k_m\|_\infty
  \sqrt{\# \Upsilon} \|\alpha\|_{\ell_2(\Upsilon)}
  + m C_m \| \beta \|_{\ell_2(\J)}, \ C_m=\max_{j\le m^2} \|\phi_j\|_\infty \\
  &\le \|k_m\|_\infty \# \Upsilon \vartheta^{-1} \bigg(1 + 2m C_m
  \sqrt{ \# \Upsilon \| G_\Upsilon^{-1}\|} \bigg)\|\widetilde \chi_\xi
  - \chi_\xi\|_\infty
\end{align*}
Moreover, using $\|\check \chi_\xi - \chi_\xi\|_\infty \le \|\check
\chi_\xi - \widetilde \chi_\xi\|_\infty + \| \chi_\xi - \widetilde
\chi_\xi\|_\infty$, we see that
\begin{equation}
\label{dist_local_lag_to_lag}
\|\check \chi_\xi -  \chi_\xi\|_\infty \le 2\| k_m\|_\infty \#
\Upsilon 
\vartheta^{-1} \bigg(1 + 2mC_m\sqrt{ \# \Upsilon \| 
  G_\Upsilon^{-1}\|} \bigg)\|\widetilde \chi_\xi - \chi_\xi\|_\infty.
\end{equation}
Finally, from Proposition~\ref{chi_tilde_properties}, if $K>(2m +
2\mu)/\nu$, it is easy to see that this holds:
\begin{equation}
  \|\check \chi_\xi -  \chi_\xi\|_\infty \le C\frac{|\log h |^2}
  {\vartheta} \left(1+ m  h^{-\mu}\right) h^{K\nu - 2m-\mu}\le 
  C\frac{|\log h |^2}{\vartheta} h^{K\nu - 2m-2\mu}  . \label{loc_lag_inf_theta}
\end{equation}

To proceed further, we need to estimate $\vartheta$. Such estimates
are known for surface splines in the Euclidean case \cite[\S6]{NW92}.
Simply repeating the proofs of \cite[Corollary~2.2]{NW92} and
\cite[Theorem~2.4]{NW92} for a set of points in $\RR^{3}$ restricted
to $\sphere$ yields the desired estimate.  For the collocation matrix
associated with $k_m$ and $\Xi$, we have
\begin{equation}
\label{theta_bound}
\vartheta \ge Cq^{2m-2},
\end{equation}
where $C$ depends only on $m$.\footnote{ For any $d$-dimensional
  sphere or projective space and any conditionally positive definite
  polyharmonic kernel with associated polynomial operator $\L_m =
  Q(-\Delta)$, where $Q$ is a polynomial of degree $m$, the
  coefficients in the expansion for $k_m$ are given by $\tilde{k}(j) =
  Q(\lambda_j)^{-1},\ j\not\in \J$. 
  For large $\lambda_j$, all of these have the asymptotic behavior
  $\lambda_j^{-m}$, which is the same as that of the coefficients for
  the $m$-$d$ thin-plate spline. This implies that the matrix $P^\perp
  \K_\Upsilon P^\perp$ in Proposition~\ref{inverse_norm_est} (here,
  $\Xi\to \Upsilon$) will have a lowest eigenvalue value that is, up
  to a constant multiple, dependent only on $m$ and $d$. Consequently,
  the bound $\vartheta \ge Cq^{2m-d}$ holds for all $k_m$ associated
  with $\L_m$ in dimension $d$.}  Thus, for $k_m$, we have $\|\check
\chi_\xi - \chi_\xi\|_\infty \le C h^{K\nu - 4m +2 - 2\mu}$, where the
constant $K$ has to be increased slightly to absorb $|\log h|^2$. With
this in mind we have the following result, whose proof, being similar
to Proposition ~\ref{chi_tilde_properties}, we omit.

\begin{theorem}\label{loc_lag_properties}
  Let the notation and assumptions of Theorem~\ref{main} hold. Suppose
  that $K>0$ is chosen so that $K>\frac{4m - 2+2\mu}{\nu}$ and, for
  each $\xi\in \Xi$, $\Upsilon(\xi) := \Xi\cap B(\xi,Kh|\log h |)$. If
  $\check \chi_\xi$ is a local Lagrange function for $\Upsilon(\xi)$
  centered at $\xi$, then set $\{\check \chi_\xi\}_{\xi\in \Xi}$ is a
  basis for $S_m(\Xi)$. Moreover, with $J:=K\nu - 4m +2 -2\mu$, we
  have
\begin{gather}
\| \check \chi_\xi - \chi_\xi\|_\infty \le C\ h^{J}, \label{error_check_chi} \\
|\check \chi_{\xi}(x)| \le C\big(1+\d(x,\xi)/h\big)^{-J}.\label{check_chi_bnd}
\end{gather}
Furthermore, when $J>d$, the set $\{\check \chi_\xi\}$ is $L_p$
stable: there are $C_1,C_2>0$ for which
\begin{equation}
\label{p_stable_check_chi_J}
C_1 q^{2/p} \|\bfa\|_{\ell_p(\Xi)}
\le 
\big\|\textstyle{\sum_{\xi\in\Xi}} a_{\xi} \check \chi_{\xi}\big\|_{L_p(\sphere)}
\le 
C_2 q^{2/p} \|\bfa\|_{\ell_p(\Xi)}.
\end{equation}
\end{theorem}

\noindent {\bf Quasi-interpolation:} It follows that the operator
$$
Q_{\Xi} f =\sum_{\xi \in \Xi} f(\xi)\check \chi_\xi 
$$
provides $L_{\infty}$ convergence at the same asymptotic rate as
interpolation $I_{\Xi}$.  Indeed,
\begin{eqnarray*}
|I_{\Xi} f (x)- Q_{\Xi} f(x) |\le \sum_{\xi\in \Xi} |\check
\chi_{\xi}(x) - \chi_{\xi}(x)| |f(\xi)|
&\le& C q^{-2}
\|f\|_{\infty}{h}^{K\nu - 4m-2\mu}\\
&\le& C\|f\|_{\infty}h^{2m}
\end{eqnarray*}
provided that $K> \frac{6m +2\mu +2}{\nu}.$ It is shown in
\cite[Corollary 5.9]{HNW2} that restricted surface spline
interpolation exhibits $\|I_{\Xi} f - f\|_{\infty} \le C h^{\sigma}$
for $f\in C^{2m}(\sphere)$ when $\sigma=2m$ and for $f\in
B_{\infty,\infty}^{\sigma}(\sphere)$ for $\sigma<2m$.  So $Q_{\Xi}$
has the same rate of approximation (without needing to solve a large
system of equations).

{\bf Constructing basis functions in terms of $N$}.  Given a set of
scattered points, it may be desirable to use $N$ as the basic
parameter instead of $h$. Therefore we wish to express the number of
nearest neighbors needed as a function of the total cardinality $N$
instead of those within a $K h\log h$ neighborhood. Considering a cap
$B(\alpha,r)$, a simple volume argument gives
\begin{equation}\label{card_est}
\#(B(\alpha,r)\cap\Xi) \le \frac{36}{11}\left(\frac{r}{q}\right)^{2}.
\end{equation}
Indeed, one arrives at this bound by first considering caps of radius
$q$ around each node in $B(\alpha,r)$. If $q$ is small enough, say
$q<r$, then at least $1/3$ the volume of each cap will be contained in
$B(\alpha,r)$. Using this and a Taylor expansion of the volume formula
$2\pi(1 - \cos(q))$ leads to \eqref{card_est}. Thus, the greatest
number of points in a cap of radius $Kh |\log h|$ is $\frac{36}{11}
\rho^2\bigl(K \log (1/h)\bigr)^2$. Also, it is not hard to show that
$2h^{-2}\le N,$ and hence it follows that the number of points is
bounded by $\frac{36}{11} \rho^2\bigl(\frac{K}{2}\log(N)\bigr)^2=
\frac{9}{11} (\rho K)^2\bigl(\log(N)\bigr)^2$, and it suffices to take
for $\Upsilon$ the nearest $\frac{9}{11} (\rho
K)^2\bigl(\log(N)\bigr)^2$ neighbors.


{\bf The constants $\nu$ and $K$}. Before we turn to a discussion of
preconditioning, we wish to comment on the constants $\nu$ and $K$
above. These two constants come into play in a crucial way in many of
our estimates.

The decay constant $\nu$ first comes up in the proof of
Theorem~\ref{main}. (Although we do not mention it in the theorem, the
proof produces two different decay constants: $\nu_L$ and $\nu_C$, the
former for the Lagrange function and the latter for the coefficients.)
The estimate for $\nu$ is, of course, a lower bound on the decay
constant itself; it is independent of $\rho$, but weakly dependent on
$m$. Because of the nature of such estimates, it is very likely that
they are much lower than $\nu_{\text{actual}}$. How
$\nu_{\text{actual}}$ behaves as a function of $\rho$ is an open
question.

There is another open question concerning $K$. We know that it must be
bounded below by $\frac{4m-2+2\mu}{\nu}$. Thus a better estimate on
$\nu$ would produce a better lower bound on $K$. This in turn means
using smaller caps and fewer points in constructing the local Lagrange
interpolant -- i.e., giving it a smaller ``footprint.'' On the other
hand, the larger we make $K$ the better the approximation to
$\chi_\xi$ we get. Since $K$ can be made as large as we please, the
question then becomes this: What is an \emph{optimal} choice for $K$?
Indeed, what does the term \emph{optimal} mean here?

\section{Preconditioning with local Lagrange functions}\label{app}

In this section we illustrate how the local Lagrange functions can
also be used as an effective preconditioner for linear systems
associated with interpolation using the standard restricted spline
basis.  Our focus is on the restricted surface spline $k_2$ (i.e. the
restricted thin plate spline), for which the interpolant to
$f\bigr|_{\Xi}$ in the standard basis takes the form
\begin{align}
  I_{\Xi} f = \sum_{\xi\in \Xi} a_{\xi} k_2(\cdot,\xi) + \sum_{j=1}^4
  c_j \phi_j(\cdot), \label{eq:std_basis}
\end{align}
where $\phi_j$ are a basis for the spherical harmonics of degree $\leq
1$.  We note that this interpolant can also be written with respect to
the local Lagrange basis for $S_2(\Xi)$ as
\begin{align}
  I_{\Xi} f = \sum_{\xi\in \Xi} {\check
    a}_{\xi}\achi_{\xi}(\cdot); \label{eq:good_basis}
\end{align}
see Section \ref{LLB} for the details on constructing this basis.  

Using the properties of the local Lagrange basis, we can write the
linear system for determining the interpolation coefficients ${\check
  a}_{\xi}$ in \eqref{eq:good_basis} as:
\begin{equation}
\begin{bmatrix}
\coll & \Phi
\end{bmatrix}
\begin{bmatrix}
\A_{\Upsilon} \\
\C_{\Upsilon}
\end{bmatrix}
\begin{bmatrix}
\vect{\check a}
\end{bmatrix}
=
\begin{bmatrix}
\vect{f}
\end{bmatrix},
\label{eq:good_lin_sys}
\end{equation}
where $(\coll)_{i,j} = k_2(\xi_i,\xi_j)$, $i,j=1,\ldots,N$, and
$\Phi_{i,j} = \phi_j(\xi_i)$, $i=1,\ldots,N$, $j= 1,\ldots, 4$. The
matrix $\A_{\Upsilon}$ is a $N$-by-$N$ \emph{sparse} matrix where each
column contains $n = M(\log N)^2$ entries corresponding to the values
of the interpolation coefficients $A_{\xi,\zeta}$ for the local
Lagrange basis in \eqref{eq:approx_lagrange}.  The matrix
$\C_{\Upsilon}$ is a $4$-by-$N$ matrix with each column containing the
values of the interpolation coefficients $c_{\xi,j}$ in
\eqref{eq:approx_lagrange}.  With the linear system written in this
way, one can view the matrix $[\A_{\Upsilon}\; \C_{\Upsilon}]^{T}$ as
a \emph{right preconditioner} for the standard kernel interpolation
matrix.  Once $\vect{\check a}$ is determined from
\eqref{eq:good_lin_sys}, we can then find the interpolation
coefficients $a_{\xi}$ and $c_j$ in \eqref{eq:std_basis} from
\begin{align}
[\vect{a}\; \vect{c}]^T = [\A_{\Upsilon}\; \C_{\Upsilon}]^{T}\vect{a}.
\end{align}

If the local Lagrange basis decays sufficiently fast then the linear
system \eqref{eq:good_lin_sys} should be ``numerically nice'' in the
sense that the matrix $\coll \A_{\Upsilon} + \Phi \C_{\Upsilon}$
should have decaying elements from its diagonal and should be well
conditioned.  As discussed in the previous section, the decay is
controlled by the number of nearest neighbors $n$ used in constructing
each local Lagrange function and that $n=M(\log N)^2$.  In the
experiments below, we found that choosing $n=7\lceil (\log N)^2/(\log
10)^2) \rceil=7\lceil (\log_{10} N)^2 \rceil$ gave very good results
over several decades of $N$.

\begin{figure}[t]
\centering
\includegraphics[width=0.55\textwidth]{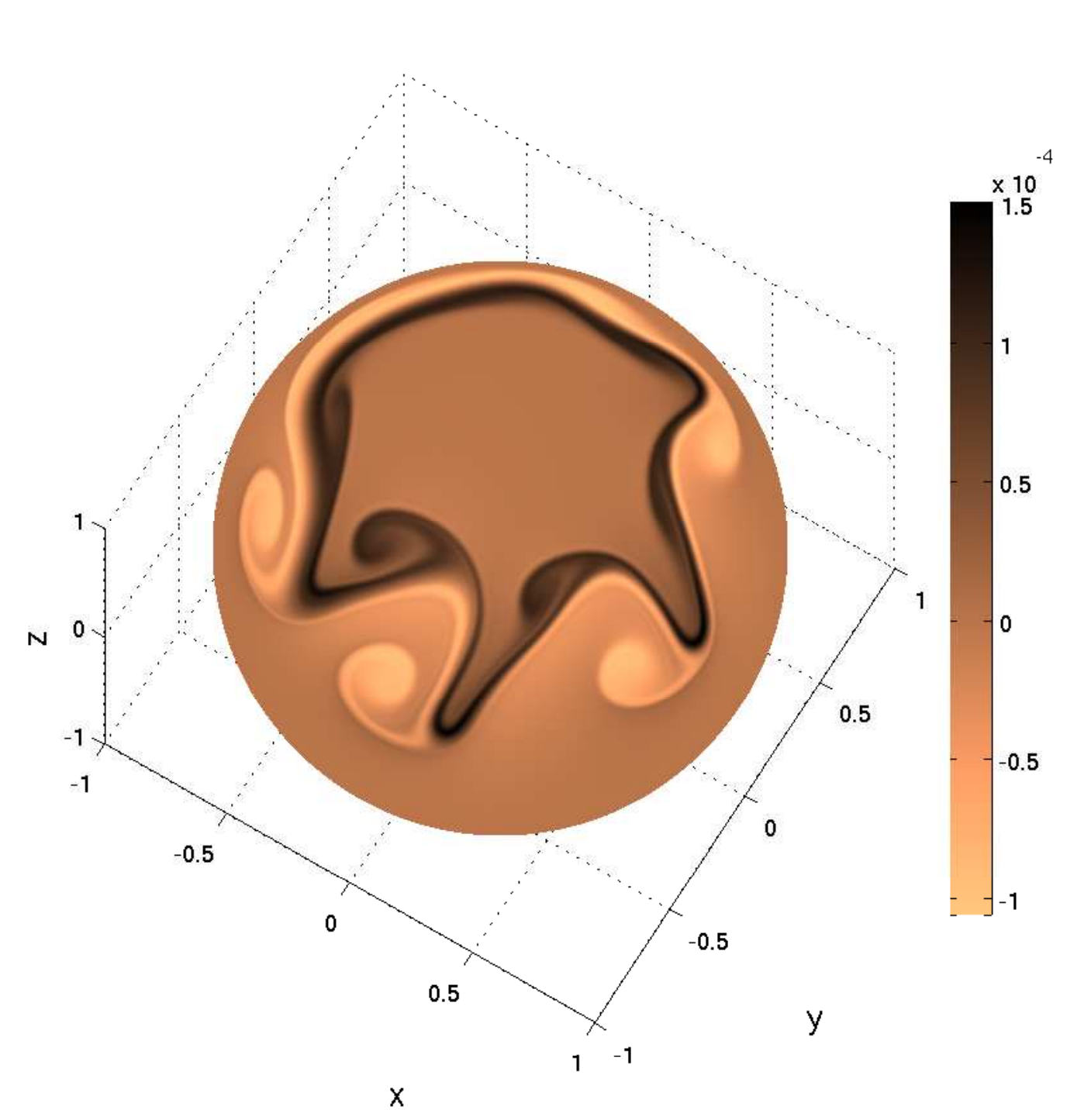}
\caption{Interpolated relative vorticity from a numerical simulation
of the shallow water wave equations on the $N=163842$ icosahedral node
sets.  The original values for the relative vorticity come
from~\cite{FlyerLehtoBlaiseWrightStCyr2012} and have been interpolated
to a regular $300\times 600$ latitude-longitude based grid using the
restricted kernel spline $k_2(x,\alpha) =
(1-x\cdot\alpha)\log(1-x\cdot \alpha)$.  The interpolation
coefficients were computed using GMRES on the preconditioned system
\eqref{eq:good_lin_sys}.\label{fig:shallow_water}}
\end{figure}

To solve the preconditioned linear system \eqref{eq:good_lin_sys} we
will use the generalized minimum residual method
(GMRES)~\cite{SaadSchultz}.  This is a Krylov subspace method which is
applicable to non-symmetric linear systems and only requires computing
matrix-vector products.  Each matrix-vector product involving the
preconditioner matrix $[A_{\Upsilon}\; \C_{\Upsilon}]^{T}$ requires
$\mathcal{O}\bigl(N(\log N)^2\bigr)$ operations, while each
matrix-vector product involving $[\coll\; \Phi]$ requires
$\mathcal{O}(N^2)$ operations.  However, Keiner \emph{et al.} have
shown that this latter product can be done in $\mathcal{O}(N\log N)$
operations using fast algorithms for spherical Fourier
transforms~\cite{Keiner:2006:FSR:1152729.1152732}.  As we are
primarily interested in the exploring the effectiveness of the local
Lagrange basis as a preconditioner, we have not used these fast
algorithms in the results below.  In a follow up study, we will
investigate these fast algorithms in combination with the
preconditioner in much more detail.

For the first numerical tests we use icosahedral node sets
$\Xi\subset\mathbb{S}^2$ of increasing cardinality.  These were chosen
because of their popularity in computational geosciences (see, for
example,~\cite{Giraldo:1997,StuhnePeltier:1999,
  Ringler:2000GeodesicGrids,Majewski:2002GME}) where interpolation
between node sets is often required.  The values of $f$ were chosen to
take on random values from a uniform distribution between $[-1,1]$.
Table \ref{tbl:iterations} displays the number of GMRES iterations to
compute an approximate solution to the resulting linear systems
\eqref{eq:good_lin_sys} for various $N$ and different tolerances.  As
we can see, the number of iterations is small and stays relatively
constant as $N$ increases.

\begin{table}[h]
\centering
\begin{tabular}{|c|c||c|c|}
\hline
& & \multicolumn{2}{c|}{Number GMRES iterations} \\
$N$ & $n$ &  $tol=10^{-6}$ & $tol=10^{-8}$ \\
\hline
\hline
2562 & 84 & 7 & 5 \\
10242 & 119 & 5 & 7 \\
23042 & 140 & 6 & 7 \\
40962 & 154 & 5 & 7 \\
92162 & 175 & 6 & 8 \\
163842 & 196 & 5 & 7\\
\hline
\end{tabular}
\caption{ Number of GMRES iterations required for computing an
approximate solution to \eqref{eq:good_lin_sys} using icosahedral node
sets of cardinality $N$.  Here $n$ corresponds to the number of nodes
used to construct the local basis and $tol$ refers to the tolerance on
the relative residual in the GMRES method.  The right hand side was
set to random values uniformly distributed between $[-1,1]$ and the
initial guess for GMRES was set equal to the function values.
\label{tbl:iterations}}
\end{table}

For the final numerical experiment, we use the above technique to
interpolate a field taken from a numerical simulation on the
icosahedral node sets to a regular latitude-longitude grid.  As
mentioned above, this is often necessary for purposes of comparing
solutions from different computational models, plotting solutions, or
coupling different models together.  The data we use comes
from~\cite{FlyerLehtoBlaiseWrightStCyr2012} and represents the
relative vorticity of a fluid described by the shallow water wave
equations on the surface of a rotating sphere.  The initial conditions
for the model lead to the development of a highly nonlinear wave with
rapid energy transfer from large to small scales, resulting in complex
vortical dynamics.  The numerical solution was computed on the
$N=163842$ node set and we interpolated it to a regular $300\times
600$ latitude-longitude based grid.  Figure \ref{fig:shallow_water}
displays the resulting interpolated relative vorticity from the
simulation at time $t=6$ days.  The figure clearly shows that the
complex flow structure has been maintained after the interpolation. As
in the numerical examples above, the approximate solution to
\eqref{eq:good_lin_sys} with this data was obtained in 7 iterations of
the GMRES method using a tolerance of $10^{-8}$.

\bibliographystyle{siam}
\bibliography{Sphere_Coeffs_submitted}
\end{document}